\documentclass{amsart}

\usepackage{amscd, amssymb, amsmath, amsthm}
\usepackage{array,caption}
\usepackage{mathabx} % for wide check 
\usepackage{enumerate, latexsym, mathrsfs}
\usepackage{xypic}
\usepackage[
		bookmarks=true, bookmarksopen=true,%
    bookmarksdepth=3,bookmarksopenlevel=2,%
    colorlinks=true,%
    linkcolor=blue,%
    citecolor=blue]{hyperref}
\newtheorem{theorem}{Theorem}[section]
\newtheorem{lemma}[theorem]{Lemma}

\newtheorem{corollary}[theorem]{Corollary}
\newtheorem{conjecture}[theorem]{Conjecture}

\theoremstyle{definition}
\newtheorem{definition}[theorem]{Definition}

\newtheorem{notation}[theorem]{Notation}

\newtheorem{example}[theorem]{Example}
\newtheorem{question}[theorem]{Question}
\newtheorem{remark}[theorem]{Remark}

\numberwithin{equation}{section}

\newcommand{\Real}{{\mathbb R}}
\newcommand{\Rational}{{\mathbb Q}}
\newcommand{\Complex}{{\mathbb C}}
\newcommand{\Integral}{{\mathbb Z}}

\title[{
Deforming abelian elliptic $\mathrm{SL}(2,\mathbb{R})$--representations of knot groups}]{
Deforming abelian elliptic\\ $\mathrm{SL}(2,\mathbb{R})$--representations of knot groups}
     
\author[Yi Liu]{%
        Yi Liu} 
\address{%
        Beijing International Center for Mathematical Research, Peking University\\
				Beijing 100871, China P.R.} 
\email{% 
    liuyi@bicmr.pku.edu.cn}

\thanks{Partially supported by NSFC Grant 12525101, 
and National Key R\&D Program of China 2020YFA0712800}

\date{% 
 \today} 

\begin{document}

\begin{abstract}
The following criterion is proved in this paper.
If the Alexander polynomial of a knot $K\subset S^3$ has a zero of odd order on the complex unit circle,
then there exists a continuous family of irreducible representations 
$\pi_1(S^3\setminus K)\to \mathrm{SL}(2,\mathbb{R})$ 
converging to an abelian representation of noncentral elliptic type.
As an application,
the author shows that the Alexander polynomial of any nontrivial L-space knot
satisfies the condition of the criterion.
In particular,
it follows that the fundamental group of any nontrivial L-space knot complement
admits an irreducible $\mathrm{SL}(2,\mathbb{R})$--representation.
\end{abstract}

\maketitle

\section{Introduction}
In this paper, 
we are concerned about $\mathrm{SL}(2,\Real)$--representations of knot groups.
For any knot $K\subset S^3$, 
we provide a criterion for the existence of 
a continuous family of irreducible representations $\pi_1(S^3\setminus K)\to\mathrm{SL}(2,\Real)$,
which deforms an abelian representation of noncentral elliptic type.
The criterion is formulated in terms of the Alexander polynomial of the knot.
We show that our criterion applies to every nontrivial L-space knot,
yielding an irreducible $\mathrm{SL}(2,\Real)$--representation of its knot group.

%Throughout this paper, we adopt the Conway normalization for the Alexander polynomial of a classical knot.
%Namely, for any (unoriented) knot $K\subset S^3$,
%the Alexander polynomial $\Delta_K\in\Integral[t,t^{-1}]$ is 
%a Laurent polynomial with the properties $\Delta_K(t)=\Delta_K(t^{-1})$ and $\Delta_K(1)=1$.
%Normalized this way, $\Delta_K$ depends only on the homeomorphism type of the knot complement $S^3\setminus K$.

We first state our main criterion, as follows.
An $\mathrm{SL}(2,\Real)$--representation of a finitely generated group is said 
to be \emph{abelian of noncentral elliptic type}
if it is conjugate in $\mathrm{SL}(2,\Real)$ to a representation 
with image contained in the subgroup $\mathrm{SO}(2)$,
but not contained in the center $\{\pm \mathbf{1}\}$.

\begin{theorem}\label{main_deform}
	For any knot $K\subset S^3$,
	if the Alexander polynomial $\Delta_K$ of $K$ has a zero of odd order on the complex unit circle,
	then $\pi_1(S^3\setminus K)$ admits 
	a continuous family of irreducible $\mathrm{SL}(2,\Real)$--representations
	which converges to 
	an abelian $\mathrm{SL}(2,\Real)$--representation of noncentral elliptic type.
	
	In particular, the family lifts to 
	a continuous, convergent family of nonabelian $\widetilde{\mathrm{SL}}(2,\Real)$--representations.
\end{theorem}

Theorem \ref{main_deform} may be viewed as an analogue 
of a similar criterion for the existence of irreducible $\mathrm{SU}(2)$--representations,
due to Heusener and Kroll \cite[Theorem 1.1]{Heusener--Kroll_abelian}, 
and due to Herald \cite[Corollary 2]{Herald_rep}.
Theorem \ref{main_deform} generalizes two known results, as follows.
If $\Delta_K$ has a simple zero on the complex unit circle,
the same conclusion as in Theorem \ref{main_deform} 
is proved by Frohman and Klassen \cite[Theorem 1.1]{Frohman--Klassen};
see also 
Heusener, Porti, and Su\'ares-Peir\'o \cite[Corollary 1.4]{HPS_reducible} for a generalization,
and see Herald and Zhang \cite[Theorem 1]{Herald--Zhang} for a stronger conclusion.
Under the extra assumption that $K\subset S^3$ is a small knot
(such as any two-bridge knot), the special case of Theorem \ref{main_deform} is 
proved recently by Dunfield and Rasmussen \cite[Corollary 1.5]{Dunfield--Rasmussen}.

See Section \ref{Sec-preliminary} for more background. 
See Theorem \ref{deform_character} and Corollary \ref{deform_representation} in Section \ref{Sec-deform} 
for a more detailed statement of Theorem \ref{main_deform}.

\begin{corollary}\label{main_deform_corollary}
	For any knot $K\subset S^3$ with determinant congruent to $3$ modulo $4$,
	or equivalently, with signature congruent to $2$ modulo $4$,
	$\pi_1(S^3\setminus K)$ admits 
	a continuous family of irreducible $\mathrm{SL}(2,\Real)$--representations
	converging to an abelian, noncentral elliptic limit.
\end{corollary}

Corollary \ref{main_deform_corollary} follows
as an immediate consequence of Theorem \ref{main_deform},
combined with well-known facts in classical knot theory.
See Corollary \ref{sgn_function_properties_corollary} 
in Section \ref{Sec-preliminary} for the relevant facts about determinant and signature.

In practice, 
the sufficient condition in Corollary \ref{main_deform_corollary}
is easy to check, and holds quite often among all knots. 
For example, 
the condition holds for arguably half of all pretzel knots (see Example \ref{pretzel_example}).

On the other hand, the condition in Corollary \ref{main_deform_corollary}
is not necessary for the Alexander polynomial to have a zero of odd order 
on the complex unit circle.
We supply a different sufficient condition, 
in terms of coefficients of the Alexander polynomial,
as follows.

\begin{theorem}\label{main_zero}
	For any knot $K\subset S^3$,
	the Alexander polynomial 
	$\Delta_K(t)=a_0+a_1\cdot(t+t^{-1})+\cdots+a_d\cdot(t^d+t^{-d})$ of $K$
	has a zero of odd order on the complex unit circle,
	if the following inequality holds for some $j\in\{1,\ldots,d\}$.
	$$|a_j|\geq |a_0|\cdot\cos\left(\frac{\pi}{\lfloor d/j\rfloor +2}\right)$$
	Here, 
	we adopt the Conway normalization for $\Delta_K\in \Integral[t,t^{-1}]$,
	namely, $\Delta_K(t)=\Delta_K(t^{-1})$, and $\Delta_K(1)=1$;
	the notation $\lfloor x\rfloor\in\Integral$ denotes the greatest integer 
	which is less than or equal to $x\in\Real$;
	one may assume $a_d\neq0$, otherwise the condition becoming stronger.
\end{theorem}

For any palindrome Laurent polynomial with real coefficients,
Konvalina and Matache show the existence of a zero on the complex unit circle,
under the similar inequality condition as in Theorem \ref{main_zero} \cite[Theorem 1]{Konvalina--Matache}
(see Theorem \ref{criterion_KM} for a restatement).
Theorem \ref{main_zero} strengthens the conclusion in the case of 
the Alexander polynomial $\Delta_K$ for any knot $K\subset S^3$,
adding the assertion about odd order.
The strengthening actually only makes use of the extra property
that $\Delta_K$ has integral coefficients and has odd value at $t=1$.
See Theorem \ref{criterion_odd_order} in Section \ref{Sec-order} 
for the slightly more general statement.

Recall that an \emph{L-space} is 
an oriented rational homology $3$--sphere $Y$ with the simplest possible Heegaard Floer homology.
Namely, $\widehat{\mathrm{HF}}(Y)$ is a finitely generated free $\Integral$--module
of rank equal to the cardinality of $H_1(Y;\Integral)$.
An oriented knot $K\subset S^3$ is called an \emph{L-space knot},
if the manifold $S^3_r(K)$ obtained by the $r$--surgery of $S^3$ along $K$ is an L-space,
for some (necessarily nonzero) coefficient $r\in\Rational$.
In this case, $K$ also admits an L-space surgery 
with some integer coefficient \cite[Corollary 1.3]{Ni_hfk_fibered} (see also \cite[Proposition 9.6]{OS_hfk_rational}).
Examples of L-space knots include all the Berge knots,
which are constructed by Berge and which admit lens space surgeries
(unpublished; see \cite[Section 1.5]{OS_hfk_lens}).

\begin{corollary}\label{L-space_knot_corollary}
	The Alexander polynomial of any nontrivial L-space knot $K\subset S^3$
	has a zero of odd order on the complex unit circle.
	Hence,
	$\pi_1(S^3\setminus K)$ admits 
	a continuous family of irreducible $\mathrm{SL}(2,\Real)$--representations
	converging to an abelian, noncentral elliptic limit.
\end{corollary}

Corollary \ref{L-space_knot_corollary} follows as a direct application of Theorems \ref{main_deform} and \ref{main_zero},
combined with well-known properties of L-space knots.
In fact, the Alexander polynomial of any L-space knot $K\subset S^3$ takes the form
\begin{equation}\label{AP_LSK}
\Delta_K(t)=\epsilon_0+\epsilon_1\cdot\left(t^{j_1}+t^{-j_1}\right)+\cdots+\epsilon_k\cdot\left(t^{j_k}+t^{-j_k}\right),
\end{equation}
for some sequence of powers $j_1,\ldots,j_k\in\Integral$ with $0<j_1<\cdots<j_k$,
and for some sequence of coefficients $\epsilon_1,\ldots,\epsilon_k\in\{\pm1\}$ with alternating signs
\cite[Corollary 1.3]{OS_hfk_lens}.
If $K$ is nontrivial, the topmost power $j_k$ (exists and) is equal to the genus of $K$,
since L-space knots are always fibered \cite[Corollary 1.3]{Ni_hfk_fibered},
and since the genus of fibered knots 
are detected by the topmost term of the Alexander polynomial.
Therefore, we obtain Corollary \ref{L-space_knot_corollary} immediately 
from Theorem \ref{main_zero} (taking $j$ to be $d=j_k$), plus Theorem \ref{main_deform}.

The actual situation is likely to be stronger,
compared to what Corollary \ref{L-space_knot_corollary} confirms.
By experiment, Culler and Dunfield notice that 
nonconstant Laurent polynomials of the form (\ref{AP_LSK}) 
appear to always have a simple zero on the complex unit circle 
\cite[Section 9, the question (5)]{Culler--Dunfield}.
If this is the case, 
then, based on the stronger criterion due to Herald and Zhang \cite[Theorem 1]{Herald--Zhang},
one will be able to confirm the following strengthening of Corollary \ref{L-space_knot_corollary}.

\begin{conjecture}\label{LSK_surgery_conjecture}
For any nontrivial L-space knot $K\subset S^3$,
and for all sufficient small (in absolute value), nonzero coefficients $r\in\Rational$,
depending on $K$,
the surgery manifold group $\pi_1(S^3_r(K))$ admits an irreducible $\mathrm{SL}(2,\Real)$--representation
which lifts to $\widetilde{\mathrm{SL}}(2,\Real)$.
\end{conjecture}

For any finitely generated $3$--manifold group,
admitting a nontrivial $\widetilde{\mathrm{SL}}(2,\Real)$--representation
is a sufficient condition for left-orderability
 \cite[Theorem 1.1]{BRW_orderable} 
(see also \cite{Bergman_orderable}).
Boyer, Gordon, and Watson conjecture
that, among all rational homology $3$--spheres,
L-spaces are precisely those with left-orderable fundamental groups,
\cite[Conjecture 1]{BGW_conjecture}.
Therefore, a proof of Conjecture \ref{LSK_surgery_conjecture} 
will support the Boyer--Gordon--Watson conjecture
in the special case of L-space knot surgeries with sufficiently small coefficients.

See Section \ref{Sec-deform} for the proof of Theorem \ref{main_deform}.
See Section \ref{Sec-order} for the proof of Theorem \ref{main_zero}.
See Section \ref{Sec-discussion} for additional comments and questions.

\subsection*{Ingredients}
Theorem \ref{main_zero} is deduced from 
the aforementioned theorem of Konvalina and Matache \cite[Theorem 1]{Konvalina--Matache},
by perturbing zeros of the given polynomial.
The perturbation trick only makes simple use of basic polynomial theory,
so we refer the reader directly to Section \ref{Sec-order} for the proof.

Below, we explain the key ideas for generating the proof of Theorem \ref{main_deform}.
To this end, we need to briefly recall the approach of Heusener and Kroll 
to the prototype of Theorem \ref{main_deform}.
Their theorem states similarly with $\mathrm{SU}(2)$--representations 
instead of $\mathrm{SL}(2,\Real)$--representations \cite[Theorem 1.1]{Heusener--Kroll_abelian}.

For any knot $K\subset S^3$, 
we can identify $\pi_1(S^3\setminus K)$ with a finitely presented group
of the form
$\Pi_\sigma=\langle x_1,\ldots x_n| x_1=\sigma(x_1),\ldots,x_n=\sigma(x_n)\rangle$,
where $\sigma$ is some element 
in the braid group $\mathscr{B}_n$ of some strand number $n$.
The group $\mathscr{B}_n$ naturally acts on the free group 
$F_n=\langle x_1,\ldots,x_n\rangle$ by automorphisms fixing 
the word $x_1\cdots x_n$.
There are two particular homomorphisms of
the free group (of rank $2n-1$)
$WF_n=\langle x_1,\ldots,x_n,y_1,\ldots,y_n|x_1\cdots x_n=y_1\cdots y_n\rangle$,
both retracting onto the obvious subgroup $F_n$.
One we denote as $r_{\mathrm{id}}\colon WF_n\to F_n$,
such that each $x_i$ and each $y_i$ map to $x_i$.
The other we denote as $r_{\sigma}\colon WF_n\to F_n$,
such that each $x_i$ maps to $x_i$, and each $y_i$ maps to $\sigma(x_i)$.
These surjective homomorphisms induce embeddings of 
the spaces of $\mathrm{SU}(2)$--characters
$$r_{\mathrm{id}}^*,r_{\sigma}^*\colon
\mathcal{X}(F_n,\mathrm{SU}(2))\to \mathcal{X}(WF_n,\mathrm{SU}(2)).$$
Therefore, $\mathcal{X}(\Pi_\sigma,\mathrm{SU}(2))$ 
can be identified as the intersection in $\mathcal{X}(WF_n,\mathrm{SU}(2))$
of the embedded images of $\mathcal{X}(F_n,\mathrm{SU}(2))$ 
via $r_{\mathrm{id}}^*$ and via $r_{\sigma}^*$.

For any $\tau\in[-2,2]$, the trace $\tau$--slice of $\mathcal{X}(\Pi_\sigma,\mathrm{SU}(2))$
refers to the closed subspace $\mathcal{X}_\tau(\Pi_\sigma,\mathrm{SU}(2))$ 
consisting of all the $\mathrm{SU}(2)$--characters $\chi\colon \Pi_\sigma\to[-2,2]$ 
with the property $\chi(x_1)=\cdots=\chi(x_n)=\tau$.
The irreducible $\mathrm{SU}(2)$--characters form an open subspace
$\mathcal{X}^{\mathtt{irr}}_\tau(\Pi_\sigma,\mathrm{SU}(2))$
in $\mathcal{X}_\tau(\Pi_\sigma,\mathrm{SU}(2))$, 
called the trace $\tau$--slice of $\mathcal{X}^{\mathtt{irr}}(\Pi_\sigma,\mathrm{SU}(2))$.
The trace $\tau$--slices for $F_n$ and for $WF_n$ are similarly defined,
by requiring $\chi(x_i)=\tau$ for each $x_i$, 
or $\chi(x_i)=\chi(y_i)=\tau$ for each $x_i$ and for each $y_i$, accordingly.
The intersection in $\mathcal{X}^{\mathtt{irr}}_\tau(WF_n,\mathrm{SU}(2))$
of the properly embedded images of $\mathcal{X}^{\mathtt{irr}}_\tau(F_n,\mathrm{SU}(2))$,
via $r_{\mathrm{id}}^*$ and via $r_{\sigma}^*$
is identified with $\mathcal{X}^{\mathtt{irr}}_\tau(\Pi_\sigma,\mathrm{SU}(2))$.

If $\tau\neq\pm2$, the trace slices $\mathcal{X}^{\mathtt{irr}}_\tau(F_n,\mathrm{SU}(2))$
and $\mathcal{X}^{\mathtt{irr}}_\tau(WF_n,\mathrm{SU}(2))$ are oriented smooth manifolds,
of dimension $2n-3$ and $4n-6$, respectively.
Moreover, for any $\theta\in(0,\pi)$, 
such that $e^{2\theta\cdot\sqrt{-1}}$ is not a complex zero of the Alexander polynomial $\Delta_K$,
the trace slice $\mathcal{X}^{\mathtt{irr}}_\tau(\Pi_\sigma,\mathrm{SU}(2))$
is compact for $\tau=2\cos\theta$.

With these facts, Heusener and Kroll count the intersections between 
the two proper embeddings of $\mathcal{X}^{\mathtt{irr}}_\tau(F_n,\mathrm{SU}(2))$ 
in $\mathcal{X}^{\mathtt{irr}}_\tau(WF_n,\mathrm{SU}(2))$, 
using transverse perturbations as usual in differential topology.
They show that the intersection number $h^\theta(K)\in\Integral$ is (well-defined and) equal to
$1/2$ times the value of the Levine--Tristram signature function 
$\mathrm{sgn}_K\colon \mathrm{U}(1)\to\Integral$ 
evaluated at $e^{2\theta\cdot\sqrt{-1}}$,
assuming $e^{2\theta\cdot\sqrt{-1}}$ is not a zero of $\Delta_K$.

If $e^{2\theta_0\cdot\sqrt{-1}}$ is an odd-order zero of $\Delta_K$ for some $\theta_0\in(0,\pi)$,
the values of $\mathrm{sgn}_K$ near $e^{2\theta_0\cdot\sqrt{-1}}$
are not equal on the different sides of $e^{2\theta_0\cdot\sqrt{-1}}$
on the complex unit circle.
This implies that 
$\mathcal{X}^{\mathtt{irr}}_\tau(\Pi_\sigma,\mathrm{SU}(2))$ must be nonempty
for some $\theta$ near $\theta_0$ and for $\tau=2\cos\theta$.
With a little more argument,
Heusener and Kroll prove the $\mathrm{SU}(2)$ prototype of Theorem \ref{main_deform}.

The similar argument with $\mathrm{SL}(2,\Real)$ in place of $\mathrm{SU}(2)$
fails at the step of intersection counting,
because the similarly defined trace slice
$\mathcal{X}_\tau(\Pi_\sigma,\mathrm{SL}(2,\Real))$
is not necessarily compact,
(so the counting could be infinite).
Dunfield and Rasmussen notice that 
the compactness of $\mathcal{X}_\tau(\Pi_\sigma,\mathrm{SL}(2,\Real))$
can be guaranteed for any small knot $K\subset S^3$.
Under the small knot assumption,
they generalize Heusener and Kroll's approach
successfully to the $\mathrm{SL}(2,\Real)$ case.
In particular, 
they obtain the $\mathrm{SL}(2,\Real)$ analogue of $h^\theta(K)$
and establish 
the similar relation with the Levine--Tristram signature function 
\cite[Theorems 1.1 and 1.2]{Dunfield--Rasmussen}.

Our proof of Theorem \ref{main_deform} relies on Heusener and Kroll's result
in a different way.
Without assuming compactness of $\mathcal{X}_\tau(\Pi_\sigma,\mathrm{SL}(2,\Real))$,
we abandon the approach of counting globally the $\mathrm{SL}(2,\Real)$ trace slice intersections.
Rather, we count them locally (and modulo $2$) near the abelian character of interest.
Our counting depends essentially on some additional choices,
including a suitably chosen open subset, 
together with suitably chosen perturbed trace slices,
such that the intersection of perturbed trace slices in the open subset is compact.
This way, it makes sense to speak of the local modulo $2$ intersection number
for the perturbed trace slices with respect to the open subset.
Due to the choices, 
the local counting does not lead to a global invariant as before.
However, the new strategy allows us 
to count the local modulo $2$ intersection numbers simultaneously
for three different types of trace slices,
namely, the $\mathrm{SL}(2,\Real)$ type, the $\mathrm{SU}(2)$ type, and the $\mathrm{SL}(2,\Complex)$ type.

These three types of local modulo $2$ intersection numbers satisfy a relation.
To explain the idea, note that we can define $\mathrm{SL}(2,\Complex)$ trace $\tau$--slices 
for any parameter $\tau\in\Complex$, and 
it works for each of the groups $F_n$, $WF_n$, and $\Pi_\sigma$.
Taking $\Pi_\sigma$ for example, 
the space of complex characters $\mathcal{X}(\Pi_\sigma,\mathrm{SL}(2,\Complex))$ 
is naturally a complex affine algebraic set.
The complex conjugation acts on it as an involution,
with fixed locus the union of
$\mathcal{X}(\Pi_\sigma,\mathrm{SL}(2,\Real))$ and $\mathcal{X}(\Pi_\sigma,\mathrm{SU}(2))$,
(which overlap along the locus of real reducible characters).
For any real trace value $\tau$, 
$\mathcal{X}^{\mathtt{irr}}_\tau(\Pi_\sigma,\mathrm{SL}(2,\Complex))$
is partitioned into 
the disjoint union of $\mathcal{X}^{\mathtt{irr}}_\tau(\Pi_\sigma,\mathrm{SL}(2,\Real))$,
and $\mathcal{X}^{\mathtt{irr}}_\tau(\Pi_\sigma,\mathrm{SU}(2))$ (empty if $\tau\not\in[-2,2])$,
and a complementary collection of non-real irreducible $\mathrm{SL}(2,\Complex)$--characters,
which are paired up by the involution.
Assume for the moment that we can maintain the similar partition in our local counting.
Then,
the resulting $\mathrm{SL}(2,\Complex)$--type local modulo $2$ intersection number
will be the modulo $2$ sum of 
the $\mathrm{SL}(2,\Real)$--type number and the $\mathrm{SU}(2)$--type number.

Near any trace value $\tau_0=2\cos\theta_0$ corresponding to an odd-order zero $e^{2\theta_0\cdot\sqrt{-1}}$ of $\Delta_K$,
we can show that the $\mathrm{SL}(2,\Complex)$--type local modulo $2$ intersection number 
is constant for all complex $\tau\neq\tau_0$ near $\tau_0$.
Along the real axis,
when $\tau$ crosses from one side of $\tau_0$ to the other side,
the $\mathrm{SU}(2)$--type local modulo $2$ intersection number jumps by $1$ modulo $2$,
as can be implied by Heusener and Kroll's theorem.
This will force a jump of the $\mathrm{SL}(2,\Real)$--type local modulo $2$ intersection number,
by $1$ modulo $2$,
which is what we need for proving Theorem \ref{main_deform}.

The most technical ingredient for carrying out the above strategy
turns out to be a suitable version of equivariant transverse perturbation,
with respect to the aformentioned involution.
We introduce a notion called $(\nu,J)$--manifolds (Definition \ref{nu_J_manifold_def}),
where $\nu$ comes originally from the complex conjugation,
and where $J$ remembers some almost-complex structure along the fixed locus.
We establish several lemmas regarding 
what we call $(\nu,J)$--equivariant transversality 
in Section \ref{Sec-transversality}.

The most crucial construction in our proof of Theorem \ref{main_deform}
lies in setting up the initial choices.
The constructive part is done in Lemma \ref{local_data}.
The compactness property for legitimating 
local modulo $2$ intersection numbers is verified in Lemma \ref{not_zero}.
The aformentioned relation 
among the three types of local modulo $2$ intersection numbers
is essentially used for proving Lemma \ref{half_signature_parity}.

\subsection*{Organization}
In Section \ref{Sec-preliminary}, 
we include more background about Theorem \ref{main_deform}.
This section is supplementary to the introduction.
The proofs of our main results 
are logically independent from Section \ref{Sec-preliminary}.

In Section \ref{Sec-braid}, 
we review braid presentation for knot groups.
In Section \ref{Sec-qreps}, 
we review $G$--representations and of $G$--characters of finitely generated groups,
for $G$ being $\mathrm{SL}(2,\Complex)$, $\mathrm{SL}(2,\Real)$, and $\mathrm{SU}(2)$,
treating from the unified perspective of quaternion algebras.
In Section \ref{Sec-trace_slices},
we discuss trace slices in spaces of $G$--characters,
for those groups arising from a braid presentation.
These three sections are somewhat expository.
They are written as preliminary for the proof of Theorem \ref{main_deform}.

In Section \ref{Sec-transversality},
we introduce $(\nu,J)$--manifolds.
We prove lemmas for disposing $(\nu,J)$--equivariant transverse perturbations.
This section builds the toolkit for our proof of Theorem \ref{main_deform}.

In Section \ref{Sec-deform}, we prove Theorem \ref{main_deform}.
The proof relies on the materials from Sections \ref{Sec-braid}--\ref{Sec-transversality}.
However,
if the reader is already familiar with Heusener and Kroll's proof for the $\mathrm{SU}(2)$ case,
it should be possible to first skim over the notations and the statements in the preceeding sections,
skipping the proofs, and then go directly to Section \ref{Sec-deform}.
This may be a quick way to access the core arguments.

In Section \ref{Sec-order}, we prove Theorem \ref{main_zero}.
This section is logically independent from other sections, and can be read directly.

In Section \ref{Sec-discussion}, we put further comments and pose open questions.

\subsection*{Acknowledgment}
The author would like to thank Nathan Dunfield for valuable communications.

\section{Background}\label{Sec-preliminary}
This section supplies as an overview on the topic
of deforming abelian representations for knot groups.
We refer the reader to Lickorish's textbook \cite{Lickorish_book_knot}
for an introduction to classical knot theory.

For any (tame, unoriented) knot $K\subset S^3$,
the \emph{Alexander polynomial} of $K$
is a Laurent polynomial in $\Integral[t,t^{-1}]$,
well-defined up to a multiplicative unit in 
$(\Integral[t,t^{-1}])^\times=\{\pm t^r\colon r\in\Integral\}$.
The Alexander polynomial depends only on the (unoriented) isotopy class of $K$.
The ambiguity can be resolved by the \emph{Conway normalization}.
Namely, there is a unique representative, denoted as
$$\Delta_K\in\Integral[t,t^{-1}],$$
which satisfies the properties $\Delta_K(t)=\Delta_K(t^{-1})$ and $\Delta_K(1)=1$.

We recall a procedure for computing the Alexander polynomial,
using a Seifert matrix.
One may take this procedure as one of the ways to define the Alexander polynomial.
Take a Seifert surface $F\subset S^3$ for $K$
(that is, an embedded orientable connected compact surface with boundary $\partial F=K$).
Fixing an auxiliary orientation for $F$,
we obtain a well-defined $\Integral$--bilinear pairing
$H_1(F;\Integral)\times H_1(F;\Integral)\to \Integral$,
called the \emph{Seifert pairing} with respect to $F$,
which is characterized by the following property.
For any oriented simple closed curves $a,b\subset \mathrm{int}(F)$, 
the Seifert pairing of $([a],[b])$ 
is equal to the linking number $\mathrm{lk}(a^-,b^+)\in\Integral$,
where $(a^-,b^+)$ denote the disjoint pair of oriented knots in $S^3$ obtained by
pushing the curves slightly off $F$, 
$a$ into the negative side and $b$ into the positive side.
Note that $H_1(F;\Integral)$ is a finitely generated free $\Integral$--module of even rank.
Fixing an auxiliary basis for $H_1(F;\Integral)$,
we obtain an integral square matrix $V$, representing the pairing.
It is known that the matrix $V-V^{\mathrm{T}}$ represents 
the intersection pairing on $H_1(F;\Integral)$.
Here, the superscript notation $\mathrm{T}$ denotes the matrix transpose.
Any matrix $V$ obtained this way is called a \emph{Seifert matrix} for $K$.

Using any Seifert matrix $V$ for $K$, 
the following formula computes the Alexander polynomial of $K$.
\begin{equation}\label{AP_def}
	\Delta_K(t)=\mathrm{det}\left(t^{1/2}\cdot V- t^{-1/2}\cdot V^{\mathrm{T}}\right)
\end{equation}
Here, the notation $\mathrm{det}$ denotes the matrix determinant.
Note that the expression 
on the right-hand side results in a Laurent polynomial in $\Integral[t,t^{-1}]$,
which satisfies the Conway normalization.
See \cite[Chapter 6]{Lickorish_book_knot}.

%Note that the complex square matrix $\omega^{1/2}\cdot V- \omega^{-1/2}\cdot V^{\mathrm{T}}$
%is Hermitian for any $\omega^{1/2}$ on the complex unit circle $U(1)\subset \Complex$.
%For any oriented knot $K\subset S^3$, 
%we assume $V$ to be obtained from a compatibly oriented Seifert surface $F$ for $K$.
%In this case, the signature of $\omega^{1/2}\cdot V- \omega^{-1/2}\cdot V^{\mathrm{T}}$

For any knot $K\subset S^3$, 
the \emph{Levine--Tristram signature function}
is a function defined on the complex unit circle with integer values, denoted as
$$\mathrm{sgn}_K\colon \mathrm{U}(1)\to \Integral.$$
The Levine--Tristram signature function depends only on the isotopy class of $K$.

Using any Seifert matrix $V$ for $K$,
the following formula computes the Levine--Tristram signature function of $K$,
evaluated at any $\omega\in\Complex$ with $|\omega|=1$.
\begin{equation}\label{sgn_omega_def}
	\mathrm{sgn}_K(\omega)=
	\begin{cases}
	\mathrm{sgn}(H(\omega))& \omega\neq1\\
	0& \omega=1
	\end{cases}
\end{equation}
Here, the notation $\mathrm{sgn}(H(\omega))$ 
denotes the signature of the Hermitian matrix
\begin{equation}\label{H_omega_def}
	H(\omega)=(1-\omega)\cdot V+(1-\bar{\omega})\cdot V^{\mathrm{T}}.
\end{equation}

We recall the following relations of the Levine--Tristram signature function $\mathrm{sgn}_K$
with the Alexander polynomial $\Delta_K$.

\begin{theorem}\label{sgn_function_properties}
	For any knot $K\subset S^3$, the following statements all hold.
	\begin{enumerate}
	\item
	The function $\mathrm{sgn}_K$
	is invariant under complex conjugation of the variable.
	\item
	The function $\mathrm{sgn}_K$ is constant
	on each open subarc of the complex unit circle 
	complementary to the zeros of $\Delta_K$.	
	\item
	The value of $\mathrm{sgn}_K$ changes by $2$ modulo $4$	
	as the variable passes across an odd-order zero of $\Delta_K$ on the complex unit circle,
	or by $0$ modulo $4$ across an even-order zero.
	\item
	The absolute value $|\Delta_K(-1)|$
	is congruent to $(-1)^{\mathrm{sgn}_K(-1)/2}$ modulo $4$.
	Note that $\Delta_K(-1)$ is an odd integer,
	and $\mathrm{sgn}_K(-1)$ is an even integer.			
	\end{enumerate}
\end{theorem}

See Section 2 in Conway's survey \cite{Conway-survey_signature}
(and in particular, Remark 2.1 and Proposition 2.3 therein) 
for the facts listed in Theorem \ref{sgn_function_properties}.
See also \cite[Chapter 8]{Lickorish_book_knot}.

For any knot $K\subset S^3$,
the \emph{signature} (or the \emph{Murasugi signature}) 
of $K$ refers to the even integer
$$\mathrm{sgn}(K)=\mathrm{sgn}_K(-1)=\mathrm{sgn}\left(V+V^{\mathrm{T}}\right).$$
The \emph{determinant} of $K$ refers to the positive odd integer
$$\mathrm{det}(K)=|\Delta_K(-1)|=\left|\mathrm{det}\left(V+V^{\mathrm{T}}\right)\right|.$$

\begin{corollary}\label{sgn_function_properties_corollary}
	For any knot $K\subset S^3$,
	if $\mathrm{det}(K)\equiv 3\bmod 4$, or equivalently,
	if $\mathrm{sgn}(K)\equiv 2\bmod 4$, then
	$\Delta_K$ admits a zero of odd order on the complex unit circle.
\end{corollary}

Corollary \ref{sgn_function_properties_corollary} follows immediately from Theorem \ref{sgn_function_properties},
by analyzing the value change of $\mathrm{sgn}_K$ along the upper complex unit semicircle.

Therefore, our quick application of the main criterion (Corollary \ref{main_deform_corollary})
follows immediately from Corollary \ref{sgn_function_properties_corollary} plus Theorem \ref{main_deform}.

\begin{example}\label{pretzel_example}
	For any $n$--tuple of nonzero integers $q_1,\ldots,q_n$ ($n\geq3$),
	denote by $P(q_1,\ldots,q_n)\subset S^3$ the $n$--pretzel link	
	with tangles of half-twist numbers $q_1,\ldots,q_n$, 
	in the customary notation (see \cite[Section 8]{Kolay_knot}).
	This is a knot if and only if either $n$ is odd and all $q_i$ are odd,
	or exactly one among all $q_i$ is even.
	The determinant of the pretzel knot is the positive odd integer computed by the following formula \cite[Theorem 17]{Kolay_knot}:
	$$\mathrm{det}(P(q_1,\ldots,q_n))=\left|q_1\cdots q_n\cdot\left(\frac{1}{q_1}+\cdots+\frac{1}{q_n}\right)\right|.$$
	The residue modulo $4$ of the determinant
	depends only on the overall sign of $q_1\cdots q_n\cdot(1/q_1+\cdots+1/q_n)$
	and the residual distribution of $q_1,\ldots,q_n$ modulo $4$.
	For example,
	assuming $n$ to be odd and $q_1,\ldots,q_n$ all to be odd and positive,
	it is an easy exercise to check that 
	$\mathrm{det}(P(q_1,\ldots,q_n))$ is congruent to $n$ modulo $4$.
	An $n$--pretzel knot is small 
	(that is, containing no essential closed surfaces in the knot complement)
	if and only if $n=3$ \cite[Corollaries 3 and 4]{Oertel_star}.
\end{example}

The complex zeros of $\Delta_K$ are closely related to 
abelian $\mathrm{SL}(2,\Complex)$--representations of $\pi_1(S^3\setminus K)$
which allow nonabelian deformations.
To review some major results in this direction,
we fix an orientation of $K$.
For simplicity, we consider any $\mathrm{GL}(1,\Complex)$--representation
\begin{equation}\label{ab_GL1C}
\rho_w\colon \pi_1(S^3\setminus K)\to \mathrm{GL}(1,\Complex),
\end{equation}
such that the oriented meridian maps to
the nonzero complex number $w$,
identifying $\mathrm{GL}(1,\Complex)=\Complex^\times=\Complex\setminus\{0\}$.
In view of the canonical subgroup inclusion 
$\mathrm{GL}(1,\Complex)\to \mathrm{SL}(2,\Complex)\colon z\mapsto \mathrm{diag}(z,z^{-1})$,
we also regard $\rho_w$ as an abelian representation 
$\pi_1(S^3\setminus K)\to \mathrm{SL}(2,\Complex)$.

\begin{theorem}\label{deform_nonabelian_facts}
	Let $K\subset S^3$ be a knot.	Fix an orientation of $K$.
	Suppose $w\in\Complex\setminus\{0\}$.	Denote $\rho_w$ as in (\ref{ab_GL1C}).
	Then, the following statements all hold.
	\begin{enumerate}
	\item
	The abelian $\mathrm{SL}(2,\Complex)$--representation $\rho_w$ is the limit of some continuous family
	of nonabelian $\mathrm{SL}(2,\Complex)$--representations,
	only if $w^2$ is a zero of $\Delta_K$.
	\item
	If $w^2$ is a zero of $\Delta_K$,
	then $\rho_w$ is the limit of some continuous family
	of reducible, nonabelian $\mathrm{SL}(2,\Complex)$--representations.
	\item
	If $w^2$ is a simple zero of $\Delta_K$,
	then $\rho_w$ is the limit of some continuous family
	of irreducible $\mathrm{SL}(2,\Complex)$--representations.
	In addition, one may require these 
	to be $\mathrm{SL}(2,\Real)$--representations, if $w\in\Real$;
	one may require these 
	to be $\mathrm{SU}(2)$--representations,
	or to be $\mathrm{SU}(1,1)$--representations, if $|w|=1$.
	\item
	If $w^2$ is a zero of $\Delta_K$ on the complex unit circle,
	such that the values of $\mathrm{sgn}_K$ near $w^2$ 
	are not equal on the different sides of $w^2$,
	then $\rho_w$ is the limit of some continuous family
	of irreducible $\mathrm{SU}(2)$--representations.	
	In particular, 
	the conclusion holds if $w^2$ is an odd-order zero of $\Delta_K$
	on the complex unit circle.
	\end{enumerate}
\end{theorem}

The statements (1) and (2) are due to Burde \cite{Burde_deform}, 
and due to de Rham \cite{de_Rham_deform}.
The statement (3) is due to Heusener, Porti, and Su\'ares-Peir\'o \cite{HPS_reducible},
generalizing a former result due to Frohman and Klassen \cite{Frohman--Klassen}.
Note that the subgroup $\mathrm{SU}(1,1)$ is conjugate to $\mathrm{SL}(2,\Real)$
in $\mathrm{SL}(2,\Complex)$, so the $\mathrm{SU}(1,1)$ case is equivalent to saying 
that some $\mathrm{SO}(2)$--representation conjugate to $\rho_w$
can be deformed into irreducible $\mathrm{SL}(2,\Real)$--representations.
The statement (4) is due to Heusener and Kroll \cite{Heusener--Kroll_abelian}, 
and due to Herald \cite{Herald_rep}.

Therefore, our main criterion (Theorem \ref{main_deform}) 
is a generalization of the $\mathrm{SU}(1,1)$ case in Theorem \ref{deform_nonabelian_facts} (3),
and is analogous to the odd-order case in Theorem \ref{deform_nonabelian_facts} (4).

\begin{remark}\label{deform_nonabelian_facts_remark}\
\begin{enumerate}
\item 
All the statements in Theorem \ref{deform_nonabelian_facts} 
hold more generally for any knot in an integral homology $3$--sphere,
(see again the aforementioned references).
Our main criterion (Theorem \ref{main_deform}) 
asserts only for classical knots,
due to technical limitation.
 %(that is, knots in $S^3$), 
To be specific,
our proof of Theorem \ref{main_deform} 
relies substantially on braid presentations.
\item
For any small knot $K\subset S^3$
the analogue of Theorem \ref{deform_nonabelian_facts} holds,
with $\mathrm{SU}(2)$ replacing $\mathrm{SL}(2,\Real)$ therein.
This follows from the recent work of Dunfield and Rasmussen \cite{Dunfield--Rasmussen}
(as the essential content of Corollary 1.5 therein).
\item
For any nontrivial knot $K\subset S^3$, 
$\pi_1(S^3\setminus K)$ always admits 
an irreducible $\mathrm{SU}(2)$--representation.
This is a theorem due to Kronheimer and Mrowka,
proved using gauge-theoretic invariants 
for smooth $4$--manifolds
\cite[Theorem 1]{Kronheimer--Mrowka_surgery}.
Note that the asserted irreducible $\mathrm{SU}(2)$--representation 
does not necessarily arise from an abelian $\mathrm{SU}(2)$--representation by deformation.
Indeed, $\Delta_K$ could have no complex zero at all 
(see \cite[Chapter 6, Corollary 6.16 (iii)]{Lickorish_book_knot}).
Whether $\pi_1(S^3\setminus K)$ always admits 
an irreducible $\mathrm{SL}(2,\Real)$--representation
remains to be an open question (communicated with Nathan Dunfield). 
\end{enumerate}
\end{remark}

\section{Braid presentations}\label{Sec-braid}
%Every knot $K\subset S^3$ can be isotoped into a tubular neighborhood of a unknot
%and as a closed braid therein. 
%This general fact guarantees 
%a braid presentation for any knot group $\pi_1(S^3\setminus K)$.
%Starting from this point, 
%it is possible to set up our problem purely in terms of group theory.

In this section, we recall braid presentations.
We prepare a collection of notations 
(Notation \ref{notation_braid}) for frequent reference in the sequel.
In the literature, 
Lin first uses braid presentations 
to give a new interpretation of the signature of a knot \cite{Lin_invariant}.
Heusener and Kroll generalize the idea
to interpret the Levine--Tristram signature function \cite{Heusener--Kroll_abelian}.

Denote by $\mathscr{B}_n$ the braid group of $n$--strands.
The standard presentation of $\mathscr{B}_n$ 
consists of $n-1$ standard generators $\sigma_1,\ldots,\sigma_{n-1}$,
and the following list of relations: for each $\mu\in\{1,\ldots,n-1\}$,
$\sigma_\mu\sigma_{\mu+1}\sigma_\mu=\sigma_{\mu+1}\sigma_\mu\sigma_{\mu+1}$,
and for each $\nu\in\{\mu+2,\ldots,n-1\}$, $\sigma_\mu\sigma_\nu=\sigma_\nu\sigma_\mu$.
The standard permutation action of $\mathscr{B}_n$ on the set $\{1,\ldots,n\}$
is completely described by the action of each generator $\sigma_\mu$,
as the transposition $(\mu,\mu+1)$.

Denote by $(X_1,\ldots,X_n)$ an alphabet of $n$ letters.
The braid group $\mathscr{B}_n$ 
acts automorphically on the freely generated group $\langle X_1,\ldots,X_n\rangle$,
such that for each $\mu,\nu\in\{1,\ldots,n-1\}$,
$$\sigma_\mu(X_\nu)=\begin{cases} 
X_\mu X_{\mu+1} X_\mu^{-1} & \nu=\mu\\
X_\mu & \nu=\mu+1\\
X_\nu & \mbox{otherwise}
\end{cases}$$
Note that the identity
$$\sigma(X_1\cdots X_n)=X_1\cdots X_n$$ 
holds in $\langle X_1,\ldots,X_n\rangle$ for all $\sigma\in\mathscr{B}_n$.
%
%For any $\sigma\in \mathscr{B}_n$, 
%the $\sigma$--cofixed group $\mathrm{Cofix}_\sigma(F_n)$ 
%refers to the quotient group of $F_n$
%by the normal closure of all the elements of the form $x^{-1}\sigma(x)$, 
%ranging over all $x\in F_n$.
%This is equivalent to expressing
%\begin{equation}\label{braid_presentation}
%\mathrm{Cofix}_\sigma(F_n)=\langle x_1,\ldots,x_n\,|\,x_1=\sigma(x_1),\ldots,x_n=\sigma(x_n)\rangle
%\end{equation}
%We refer to the group presentation 
%appearing on the right-hand side of (\ref{braid_presentation}) 
%as the \emph{braid presentation} associated to $\sigma\in\mathscr{B}_n$
%on the alphabet $(x_1,\ldots,x_n)$.

\begin{notation}\label{notation_braid}
	Denote by $\mathscr{B}_n$ the braid group of $n$--strands.
	\begin{enumerate}
	\item
	For any $\sigma\in\mathscr{B}_n$,	denote
	$$\Pi_\sigma=\langle x_1,\ldots,x_n\,|\,x_1=\sigma(x_1),\ldots,x_n=\sigma(x_n)\rangle.$$
	Here, the generators are treated as distinguished elements in $\Pi_\sigma$,
	and the relations as equations in $\Pi_\sigma$;
	the shorthand notation $\sigma(x_\mu)$ denotes the word $\sigma(X_\mu)$ 
	substituting $(X_1,\ldots,X_n)$ with $(x_1,\ldots,x_n)$.
	\item
	Denote
	$$F_n=\langle x_1,\ldots,x_n\rangle,$$
	and
	$${WF}_n=\langle x_1,\ldots,x_n,y_1,\ldots,y_n\,|\,x_1\cdots x_n=y_1\cdots y_n\rangle.$$
	%Note that $F_n$ is free of rank $n$, and ${WF}_n$ is free of rank $2n-1$.
	\item
	For any $\sigma\in\mathscr{B}_n$,
	denote by %we introduce the group homomorphisms
	$$p_\sigma\colon {WF}_n\to \Pi_\sigma$$
	the group homomorphism assigning $p_\sigma(x_\mu)=x_\mu$ and $p_\sigma(y_\mu)=x_\mu$;
	denote by
	$$r_\sigma\colon {WF}_n\to F_n$$
	the group homomorphism assigning $r_\sigma(x_\mu)=x_\mu$ and $r_\sigma(y_\mu)=\sigma(x_\mu)$
	%Note that $p_\sigma$ and $r_\sigma$ are both surjective.
	\end{enumerate}
\end{notation}

The following lemma describes a general picture
that lies behind Notation \ref{notation_braid}.
It says that
$\mathrm{Hom}(\Pi_n,G)$ can be studied, 
via $p_\sigma^*$,
as the intersection in $\mathrm{Hom}({WF}_n,G)$
of two embeddings of $\mathrm{Hom}(F_n,G)$,
namely, a ``diagonal embedding'' via $r_{\mathrm{id}}^*$,
and a ``graph embedding'' via $r_\sigma^*$.
The special cases with $G$ being 
$\mathrm{SU}(2)$, $\mathrm{SL}(2,\Real)$, and $\mathrm{SL}(2,\Complex)$
are of primary interest to us in the sequel.

\begin{lemma}\label{hom_G_picture}
	Let $G$ be any group. Adopt Notation \ref{notation_braid}.
	\begin{enumerate}
	\item
	For any $\sigma\in\mathscr{B}_n$,
	the maps induced by group homomorphisms
	$$p_\sigma^*\colon\mathrm{Hom}(\Pi_\sigma,G)\to \mathrm{Hom}({WF}_n,G)$$
	and
	$$r_\sigma^*,r_{\mathrm{id}}^*\colon \mathrm{Hom}(F_n,G)\to \mathrm{Hom}({WF}_n,G)$$
	are all injective.
	%\begin{eqnarray*}
	%p^*&\colon&\mathrm{Hom}(\Pi_\sigma,G)\to \mathrm{Hom}({WF}_n,G)\\
	%r_\sigma^*,r_{\mathrm{id}}^*&\colon& \mathrm{Hom}(F_n,G)\to \mathrm{Hom}({WF}_n,G)
	%\end{eqnarray*}
	\item
	The following identity of subsets holds in $\mathrm{Hom}({WF}_n,G)$.
	$$p_\sigma^*(\mathrm{Hom}(\Pi_\sigma,G))=r_\sigma^*(\mathrm{Hom}(F_n,G))\cap r_{\mathrm{id}}^*(\mathrm{Hom}(F_n,G))$$
	\end{enumerate}
\end{lemma}

\begin{proof}
	The group homomorphisms $r_{\mathrm{id}},r_\sigma\colon {WF}_n\to F_n$, 
	and $p_\sigma\colon {WF}_n\to \Pi_\sigma$ are all surjective.
	Denote by $q_\sigma\colon F_n\to \Pi_\sigma$ the group homomorphism
	assigning $q_\sigma(x_\mu)=x_\mu$, which is also surjective.
	Observe $p_\sigma=q_\sigma\circ r_{\mathrm{id}}=q_\sigma\circ r_\sigma$.
	Applying the functor $\mathrm{Hom}(\Box,G)$,
	the assertions follow immediately.	
\end{proof}

The presentation of $\Pi_\sigma$ appearing in Notation \ref{notation_braid}
is called a \emph{braid presentation}.
Braid presentations have been frequently used in classical knot theory.
We include a brief proof of the following well-known fact
for the reader's reference.

\begin{lemma}\label{make_braid_presentation}
	For any oriented link $L\subset S^3$, 
	and for any sufficiently large $n$ depending on $L$,
	there exists some braid group element $\sigma\in\mathscr{B}_n$,
	such that	the fundamental group of the link complement
	$\pi_1(S^3\setminus L)$ 
	is isomorphic to the presented group $\Pi_\sigma$ as in Notation \ref{notation_braid},
	and such that the conjugacy class of each generator $x_1,\ldots,x_n$ 
	corresponds to the free homotopy class of the oriented meridian around some component of $L$.
	In this way, the components of $L$ correspond bijectively to 
	the permutation orbits of $\sigma$ acting on $\{1,\ldots,n\}$.
\end{lemma}

\begin{proof}
	Isotope $L$ into a tubular neighborhood of a unknot as a closed braid therein.
	The monodromy of the braid in solid torus can be regarded as 
	the mapping torus of orientation-preserving, boundary-fixing self-diffeomorphism 
	of a disk with $n$ punctures, where $n$ denotes the number of strands.
	The mapping class (that is, boundary-fixing isotopy class) of the monodromy
	can be regarded as a braid group element $\sigma\in \mathscr{B}_n$.
	The link complment $S^3\setminus L$ homeomorphic to the Dehn filling of 
	the mapping torus such that 
	the suspension flow orbits bound disk in the filling solid torus.
	The construction above yields 
	$\pi_1(S^3\setminus L)\cong\langle x_1,\ldots,x_n,t\,|
	\,t^{-1}x_1t=\sigma(x_1),\ldots,t^{-1}x_nt=\sigma(x_n),t=1\rangle
	\cong \Pi_\sigma$, as asserted.
\end{proof}

We close this section by mentioning a topological interpretation
of the groups appearing in Notation \ref{notation_braid}.
The topological interpretation is not needed in the sequel,
but it may explain where the groups come from.

For any $\sigma\in\mathscr{B}$,
consider a braid $b_\sigma$ (representing $\sigma$) 
as $n$ properly embedded, mutually disjoint strands in $D^3$,
such that the initial endpoints of the strands 
are $n$ (unordered) prescribed points $P^-=\{p^-_1,\ldots,p^-_n\}$ 
on (the southern hemisphere of) $\partial D^3=S^2$,
and the terminal points are another $n$ prescribed points $P^+=\{p^+_1,\ldots,p^+_n\}$ 
on (the northern hemisphere of) $S^2$.
%We assume that the $n$ prescribed initial endpoints 
%lie in the interior of the southern hemisphere $D^2_-$,
%and the $n$ prescribed terminal points lie in the interior of the northern hemisphere $D^2_+$,
%with respect to the decomposition of $S^2=D^2_-\cup_{S^1} D^2_+$ along the equator $S^1$.
Take another copy of $D^3$ containing 
a trivial $n$--strand braid $b_{\mathrm{id}}$ with prescribed endpoints,
and glue it with the copy containing $b_\sigma$ by identifying their boundaries $S^2$.
In the resulting $S^3$,
the union of $b_{\mathrm{id}}$ and $b_{\sigma}$ forms 
a closed braid (as a link) $L=\hat{b}_\sigma$.
By construction, the link complement $S^3\setminus L$
decomposes along the gluing locus $S^2\setminus(P^+\cup P^-)$
into two braid complements 
$D^3\setminus b_\sigma\approx D^3\setminus b_{\mathrm{id}}$.

By fixing a basepoint in $S^2\setminus (P^+\cup P^-)$,
the group ${WF}_n$ in Notation \ref{notation_braid}
can be identified with $\pi_1(S^2\setminus(P^+\cup P^-))$,
such that the generators $x_\mu$ and $y_\mu$ are represented by 
closed paths going around $p^-_\mu$ and $p^+_\mu$,
respectively.
The group $F_n$ can be identified with 
$\pi_1(D^3\setminus b_{\mathrm{id}})\cong\pi_1(D^3\setminus b_{\sigma})$.
The group $\Pi_\sigma$ is identified with $\pi_1(S^3\setminus L)$,
in accordance with the van Kampen theorem.
The group homomorphisms  $r_{\sigma}$, $r_{\mathrm{id}}$, and $p_\sigma$
appearing in Lemma \ref{hom_G_picture} are precisely as induced
by obvious inclusions of topological spaces.

\section{Quaternionic representations}\label{Sec-qreps}
In this section,
we collect preliminary facts regarding 
the spaces of representations in the linear groups 
$\mathrm{SU}(2)$, $\mathrm{SL}(2,\Real)$, and $\mathrm{SL}(2,\Complex)$,
and regarding the spaces of their characters.
For these linear groups, it is possible to take a unified approach, 
via real or complex quaternion algebras.

For a general introduction to quaternion algebras, 
we refer to Maclachlan and Reid's textbook \cite{Maclachlan--Reid_book},
(see Chapters 2, 6, and 7 therein).
We only need to recall a few terms and basic facts for our exposition.

Throughout this section, we work over (commutative) fields of characteristic $0$.

\subsection{Quaternion algebras}
Let $\mathbb{F}$ be a field of characteristic $0$.
In constructive terms, 
a quaternion algebra $\mathscr{A}$ over $\mathbb{F}$
is a $4$--dimensional vector space over $\mathbb{F}$ of the form
$$\mathscr{A}=\mathbb{F}\mathbf{1}\oplus \mathbb{F}\mathbf{i}\oplus\mathbb{F}\mathbf{j}\oplus\mathbb{F}\mathbf{k},$$
such that 
$\mathbf{1}$ is the identity of $\mathscr{A}$,
and such that the (associative) multiplication on $\mathscr{A}$ 
is determined by the rules
$\mathbf{k}=\mathbf{i}\mathbf{j}=-\mathbf{j}\mathbf{i}$,
and $\mathbf{i}^2=a\mathbf{1}$, $\mathbf{j}^2=b\mathbf{1}$,
for some $a,b\in\mathbb{F}^\times$.
We obtain $\mathbf{k}^2=-ab\mathbf{1}$, and 
$-a\mathbf{j}=\mathbf{k}\mathbf{i}=-\mathbf{i}\mathbf{k}$,
and
$-b\mathbf{i}=\mathbf{j}\mathbf{k}=-\mathbf{k}\mathbf{j}$.
A quaternion algebra $\mathscr{A}$ as above is said to have 
Hilbert symbol $(a,b)$ over $\mathbb{F}$,
with respect to the basis $\mathbf{1},\mathbf{i},\mathbf{j},\mathbf{k}$.
A Hilbert symbol of $\mathscr{A}$ depends not only on
the $\mathbb{F}$--algebra isomorphism type of $\mathscr{A}$,
but also on the choice of a basis.

The center of $\mathscr{A}$ is the $\mathbb{F}$--subalgebra $\mathbb{F}\mathbf{1}$,
identified with $\mathbb{F}$ henceforth.
The $\mathbb{F}$--linear direct summand $\mathbb{F}\mathbf{i}\oplus\mathbb{F}\mathbf{j}\oplus\mathbb{F}\mathbf{k}$
can be uniquely characterized, independent of the choices of $\mathbf{i},\mathbf{j},\mathbf{k}$,
by the property that any nonzero element therein is square-central but not central.
Therefore, 
there is a well-defined operation $\mathscr{A}\to\mathscr{A}\colon q\mapsto\bar{q}$,
called the \emph{quaternion conjugation},
as uniquely determined by the expression:
\begin{equation}\label{q_conjugation_def}
\overline{t\mathbf{1}+x\mathbf{i}+y\mathbf{j}+z\mathbf{k}}=t\mathbf{1}-x\mathbf{i}-y\mathbf{j}-z\mathbf{k}
\end{equation}

There are natural functions 
$\mathrm{Nr}\colon \mathscr{A}\to \mathbb{F}$ 
and $\mathrm{Tr}\colon \mathscr{A}\to \mathbb{F}$,
called the \emph{norm} and the \emph{trace}, respectively.
They are defined by the expressions:
\begin{equation}\label{Nr_def}
\mathrm{Nr}(q)=q\bar{q}
\end{equation}
\begin{equation}\label{Tr_def}
\mathrm{Tr}(q)=q+\bar{q}
\end{equation}
%We often drop the subscript if $\mathscr{A}/\mathbb{F}$ is clear from the context.
For any $q\in\mathscr{A}$, the following identity holds in $\mathscr{A}$:
$$q^2-\mathrm{Tr}(q)\,q+\mathrm{Nr}(q)=0.$$
In particular, $q$ is invertible if and only if $\mathrm{Nr}(q)$ is nonzero,
and in that case, $q^{-1}$ is equal to $\mathrm{Nr}(q)^{-1}(\mathrm{Tr}(q)-\bar{q})$.

We introduce the following notations.
Note that $\mathrm{SL}(1,\mathscr{A})$ forms a linear $\mathbb{F}$--algebraic group,
and $\mathfrak{sl}(1,\mathscr{A})$ is the associated $\mathbb{F}$--Lie algebra.
\begin{equation}\label{SL_def}
\mathrm{SL}(1,\mathscr{A})=\{q\in\mathscr{A}\colon \mathrm{Nr}(q)=1\}
\end{equation}
\begin{equation}\label{SL_Lie_def}
\mathfrak{sl}(1,\mathscr{A})=\{q\in\mathscr{A}\colon \mathrm{Tr}(q)=0\}
\end{equation}

%\begin{remark}
	%As a justification for our notation,
	%the linear $\mathbb{F}$--algebraic group $\mathrm{SL}(r,\mathscr{A})$ and 
	%the $\mathbb{F}$--Lie algebra $\mathfrak{sl}(r,\mathscr{A})$ makes sense for any rank $r$.
	%To be precise, denote by $\mathrm{M}_r(\mathscr{A})$ the $\mathbb{F}$--algebra of all square matrices
	%of size $m$ and	with entries in $\mathscr{A}$.
	%The linear $\mathbb{F}$--algebraic group $\mathrm{GL}(r,\mathscr{A})$ 
	%consists of all invertible $B\in\mathrm{M}_r(\mathscr{A})$,
	%with Lie algebra $\mathfrak{gl}(r,\mathscr{A})=\mathrm{M}_r(\mathscr{A})$.
	%Then, 
	%$\mathfrak{sl}(r,\mathscr{A})$ is	$\mathbb{F}$--Lie subalgebra of $\mathfrak{gl}(r,\mathscr{A})$
	%consisting of all the trace $0$ matrices, where trace refers to the sum of $\mathrm{Tr}$ for all the diagonal entries;
	%$\mathrm{SL}(m,\mathscr{A})$ is the $\mathbb{F}$--algebraic subgroup of $\mathrm{GL}(r,\mathscr{A})$
	%corresponding to $\mathfrak{sl}(r,\mathscr{A})$.
%\end{remark}

\begin{example}\label{q_algebra_example}\
There is a unique quaternion algebra over the field of complex numbers $\Complex$,
up to isomorphism of complex algebras.
It is the complex algebra $\mathrm{M}_2(\Complex)$ 
of complex square matrices of size $2$.
(See \cite[Chapter 2, Theorem 2.3.1]{Maclachlan--Reid_book}.)
There are exactly two quaternion algebras over the field of real numbers $\Real$,
up to isomorphism of real algebras.
One is the matrix algebra $\mathrm{M}_2(\Real)$.
The other is the Hamilton quaternion algebra $\mathcal{H}$,
which we identify as the real subalgebra of $\mathrm{M}_2(\Complex)$ 
spanned $\Real$--linearly by the identity matrix and 
the three standard Pauli matrices.
(See \cite[Chapter 2, Theorem 2.5.1]{Maclachlan--Reid_book}.)
In all these cases, the quaternion trace agrees with the matrix trace, 
and the quaternion norm agrees with the matrix determinant.
The corresponding groups $\mathrm{SL}(1,\mathscr{A})$ 
and Lie algebras $\mathfrak{sl}(1,\mathscr{A})$ are summarized as in Table \ref{q_algebra_table}.
\begin{table}[h]
		\bigskip
		\centering
		{%\footnotesize
		\begin{tabular}{|m{1.5cm}|m{1.5cm}|m{3cm}|m{3cm}|}
		\hline
		$\mathbb{F}$ & $\mathscr{A}$ & $\mathrm{SL}(1,\mathscr{A})$ & $\mathfrak{sl}(1,\mathscr{A})$   \\ [0.5ex]
		\hline
		$\Complex$ & $\mathrm{M}_2(\Complex)$  & $\mathrm{SL}(2,\Complex)$ & $\mathfrak{sl}(2,\Complex)$  \\ 
		$\Real$ & $\mathrm{M}_2(\Real)$ &  $\mathrm{SL}(2,\Real)$ & $\mathfrak{sl}(2,\Real)$  \\
		$\Real$  & $\mathcal{H}$ & $\mathrm{SU}(2)$ & $\mathfrak{su}(2)$ \\ [1ex]
		\hline
		\end{tabular}
		}
		\caption{Classification of complex and real quaternion algebras}\label{q_algebra_table}
\end{table}	
\end{example}

For any quaternion algebra $\mathscr{A}$ over any field $\mathbb{F}$ of characteristic $0$,
and for any $\alpha,\beta\in\mathrm{SL}(1,\mathscr{A})$,
it is straightforward to check the following identity in $\mathbb{F}$.
	$$\mathrm{Tr}(\alpha)\mathrm{Tr}(\beta)=\mathrm{Tr}(\alpha\beta)+\mathrm{Tr}(\alpha\beta^{-1})$$
This basic relation generates many more sophisticated ones,
giving rise to a collection of universal trace polynomials.
We provide a formal statement of this fundamental fact, as the following lemma.

\begin{lemma}\label{trace_polynomial}
	For any alphabet $(X_1,\ldots,X_n)$ of $n$ letters,
	and for some sufficiently large $N$ depending on $n$,
	there exists a finite tuple $(y_1,\ldots,y_N)$ of words,
	as elements in the freely generated group $\langle X_1,\ldots,X_n\rangle$,
	such that the following property holds.
	
	For any word $w=w(X_1,\ldots,X_n)$,
	there exists some $N$--variable polynomial $P_w=P_w(T_1,\ldots,T_N)$ over $\Integral$,
	such that 
	for any quaternion algebra $\mathscr{A}$ over	any field $\mathbb{F}$ of characteristic $0$,
	and	for any $N$-tuple $(q_1,\ldots,q_N)$ of elements in $\mathrm{SL}(1,\mathscr{A})$,
	the following identity holds in $\mathbb{F}$.
	$$\mathrm{Tr}(w(q_1,\ldots,q_n))=
	P_w\left(\mathrm{Tr}(y_1(q_1,\ldots,q_n)),\ldots,\mathrm{Tr}(y_N(q_1,\ldots,q_n))\right)$$
\end{lemma}

\begin{proof}
	This follows essentially from the same argument as in \cite[Proposition 1.4.1]{Culler--Shalen_variety},
	by repeatedly using trace relations. 
	Due to the slightly different setting and formulation,
	we outline the argument and clarify some key points.
	
	Following the same procedure as in the proof of \cite[Proposition 1.4.1]{Culler--Shalen_variety}, 
	for any finitely generated subgroup 
	$\Gamma$ of $\mathrm{SL}(1,\mathscr{A})$ and for any finite tuple of generators
	$\gamma_1,\ldots,\gamma_n$ of $\Gamma$,
	one can express the trace $\mathrm{Tr}(\theta)$ in $\mathbb{F}$ of any element 
	$\theta\in\Gamma$
	as a polynomial in $\mathrm{Tr}(\gamma_{i_1}\cdots\gamma_{i_r})$ over $\Integral$,
	where $(i_1,\ldots,i_r)$ ranges over 
	all finite tuples with mutually distinct $i_1,\ldots,i_r\in\{1,\ldots,n\}$.
	Moreover, a polynomial $P_w=P_w(T_1,\ldots,T_N)$ 
	with this property can be constructed algorithmically, 
	such that the only input is a free word $w=w(X_1,\ldots,X_n)$
	that realizes $\theta$ as $w(\gamma_1,\ldots,\gamma_n)$ in $\Gamma$.
	Here, $N$ denotes the number of tuples $(i_1,\ldots,i_r)$ as above,
	which is equal to $n!\cdot (1/(n-1)!+\cdots+1/2!+1/1!)$.
	In particular, $P_w$ does not depend on $\mathscr{A}/\mathbb{F}$ or 
	on any relation of the generators in $\Gamma$.
	Take the words $X_{i_1}\cdots X_{i_r}$ in $\langle X_1,\ldots,X_n\rangle$
	ranging over all $(i_1,\ldots,i_r)$ as above,
	and enumerate them as $y_1,\ldots,y_N$.
	The above description shows that the tuple of words $(y_1,\ldots,y_N)$ 
	is as asserted.	
\end{proof}

\begin{remark}
	A more precise statement for the property in Lemma \ref{trace_polynomial} is that
	there exists 
	some computable function $w\mapsto P_w\colon \langle X_1,\ldots,X_n\rangle \to \Integral[T_1,\ldots,T_N]$
	satisfying the asserted identity.
	Here, being computable means that there exists some program (or a Turing machine),
	such that for every valid input $w$, the computation halts with output $P_w$.
	The number $N$ as obtained in the proof is certainly not optimally small in general.
\end{remark}

%
%for any $a,b\in \mathbb{F}^\times$, 
%a quaternion algebra $\mathscr{A}$ over $\mathbb{F}$ with Hilbert symbol %$(a,b/\mathbb{F})$
%$$\left(\frac{a,b}{\mathbb{F}}\right)$$
%is isomorphic to
%$\mathbb{F}\mathbf{1}\oplus \mathbb{F}\mathbf{i}\oplus\mathbb{F}\mathbf{j}\oplus\mathbb{F}\mathbf{k}$
%as a $4$--dimensional vector space over $\mathbb{F}$,
%such that the multiplication on $\mathscr{A}$ 
%is determined by the rules
%$\mathbf{k}=\mathbf{i}\mathbf{j}=-\mathbf{j}\mathbf{i}$,
%and $\mathbf{i}^2=a\mathbf{1}$, $\mathbf{j}^2=b\mathbf{1}$,

\subsection{Spaces of representations}
Let $\mathscr{A}$ be a quaternion algebra over a field $\mathbb{F}$
of characteristic $0$. Let $\Pi$ be a finitely generated group.

A \emph{representation} of $\Pi$ in $\mathrm{SL}(1,\mathscr{A})$
refers to any group homomorphism $\Pi\to \mathrm{SL}(1,\mathscr{A})$.
We say that a representation $\rho\colon\Pi\to\mathrm{SL}(1,\mathscr{A})$
is \emph{irreducible} if the $\mathbb{F}$--subalgebra of $\mathscr{A}$ 
spanned by the image of $\rho$ is $\mathscr{A}$, 
or otherwise \emph{reducible}.
If $\mathbb{F}$ is algebraically closed, 
$\mathscr{A}$ is isomorphic to the $\mathbb{F}$--matrix algebra $\mathrm{M}_2(\mathbb{F})$,
and the irreducibility is equivalent to the usual notion
that $\rho\colon \Pi\to \mathrm{SL}(2,\mathbb{F})$ preserves 
some $1$--dimensional subspace of the $\mathbb{F}$--vector space $\mathbb{F}^2$.

Fix a generating tuple $(g_1,\ldots,g_n)$ of $\Pi$.
We obtain a map of sets
$$\mathrm{Hom}(\Pi,\mathrm{SL}(1,\mathscr{A}))\to \mathscr{A}^n
\colon \rho \mapsto (\rho(g_1),\ldots,\rho(g_n))$$
This map is clearly injective. 
We denote the image of the above map as
$$\mathcal{R}(\Pi,\mathrm{SL}(1,\mathscr{A}))\subset\mathscr{A}^n,$$
and identify the image with $\mathrm{Hom}(\Pi,\mathrm{SL}(1,\mathscr{A}))$,
with respect to the generating tuple.
%$$\mathcal{R}(\Pi,\mathrm{SL}(1,\mathscr{A}))
%=\left\{
%(\rho(g_1),\ldots,\rho(g_n))\in\mathscr{A}^n\colon
%\rho\in\mathrm{Hom}(\Pi,\mathrm{SL}(1,\mathscr{A}))\right\},$$
%with respect to the generating tuple, and viewed as a subset of $\mathscr{A}^n$.
We decorate $\mathcal{R}$ with superscripts $\mathtt{irr}$ or $\mathtt{red}$ 
to denote the subsets of irreducible or reducible representations,
respectively.
Hence, we obtain a partition of sets
$$\mathcal{R}(\Pi,\mathrm{SL}(1,\mathscr{A}))=
\mathcal{R}^{\mathtt{irr}}(\Pi,\mathrm{SL}(1,\mathscr{A}))\sqcup 
\mathcal{R}^{\mathtt{red}}(\Pi,\mathrm{SL}(1,\mathscr{A})).$$
We refer to $\mathcal{R}(\Pi,\mathrm{SL}(1,\mathscr{A}))$ 
as the \emph{space of representations} of $\Pi$ in $\mathrm{SL}(1,\mathscr{A})$.

This space of representations $\mathcal{R}(\Pi,\mathrm{SL}(1,\mathscr{A}))$
is a (Zariski closed) algebraic set 
sitting in the affine $\mathbb{F}$--linear space $\mathscr{A}^n\cong \mathbb{F}^{4n}$.
In fact, $\mathcal{R}(\Pi,\mathrm{SL}(1,\mathscr{A}))$ is 
the set of solutions $(q_1,\ldots,q_n)$ in $\mathscr{A}^n$
to the system of equations $\mathrm{Nr}(q_1)=\cdots=\mathrm{Nr}(q_n)=1$,
and $r(q_1,\ldots,q_n)=\mathbf{1}$, 
ranging over all relators $r=r(X_1,\ldots,X_n)$
for $\Pi$ with respect to the generating tuple.
Fixing a $\mathbb{F}$--linear basis 
$\mathbf{1},\mathbf{i},\mathbf{j},\mathbf{k}$ as before,
these equations reduce to 
a system of polynomial equations with $\Integral$ coefficients
in the coordinates $(t_1,x_1,y_1,z_1,\ldots,t_n,x_n,y_n,z_n)$ in $\mathbb{F}^{4n}$.
For different choices of the generating tuple,
the resulting spaces of representations are 
naturally isomorphic to each other as $\mathbb{F}$--algebraic sets.

The subspace $\mathcal{R}^{\mathtt{red}}(\Pi,\mathrm{SL}(1,\mathscr{A}))$
of reducible representation is Zariski closed
in $\mathcal{R}(\Pi,\mathrm{SL}(1,\mathscr{A}))$.
Hence, $\mathcal{R}^{\mathtt{irr}}(\Pi,\mathrm{SL}(1,\mathscr{A}))$
is Zariski open in $\mathcal{R}(\Pi,\mathrm{SL}(1,\mathscr{A}))$.
In fact, a reducible representation $\rho\colon\Pi\to \mathrm{SL}(1,\mathscr{A})$
can be characterized equivalently by the property
$\mathrm{Tr}(\rho([g,h]))=2$ for all $g,h\in\Pi$,
where $[g,h]$ denotes the commutator $ghg^{-1}h^{-1}$.
This means that the subset $\mathcal{R}^{\mathtt{red}}(\Pi,\mathrm{SL}(1,\mathscr{A}))$
of $\mathcal{R}(\Pi,\mathrm{SL}(1,\mathscr{A}))$ is cut out by a system of equations
$\mathrm{Tr}([u(q_1,\ldots,q_n),v(q_1,\ldots,q_n)])=2$, 
ranging over all words $u=u(X_1,\ldots,X_n)$ and $v=v(X_1,\ldots,X_n)$.
These equations reduce to 
a system of polynomial equations with $\Integral$ coefficients
in the coordinates of $\mathbb{F}^{4n}$.

\subsection{Spaces of characters}
Let $\mathbb{F}$ be a field of characteristic $0$.
Let $\Pi$ be a finitely generated group.

For any quaternion algebra $\mathscr{A}$ over $\mathbb{F}$,
and for any representation $\rho\colon \Pi\to\mathrm{SL}(1,\mathscr{A})$,
the \emph{character} of $\rho$
refers to the function $\chi_\rho\colon \Pi\to \mathbb{F}$
defined by $\chi_\rho(g)=\mathrm{Tr}(\rho(g))$ for all $g$.
The function $\chi_\rho$ is constant on any conjugacy class of $\Pi$.
For any other representation $\rho'\colon\Pi\to\mathrm{SL}(1,\mathscr{A})$
with $\chi_{\rho'}$ identical to $\chi_{\rho'}$,
$\rho'$ is irreducible if and only if $\rho$ is irreducible,
and in this case, $\rho$ is conjugate to $\rho'$ 
by some element in $\mathrm{GL}(1,\mathscr{A})=\mathscr{A}^\times$.
In particular, it makes sense to speak of
\emph{irreducible} characters and \emph{reducible} characters.

Fix a generating tuple $(g_1,\ldots,g_n)$ of $\Pi$.
Fix a tuple of words $(y_1,\ldots,y_N)$ 
as guaranteed in Lemma \ref{trace_polynomial}.
Obtain a tuple $(h_1,\ldots,h_N)$ of elements in $\Pi$,
by setting $h_1=y_1(g_1,\ldots,g_n),\ldots,h_N=y_N(g_1,\ldots,g_n)$.
We obtain a regular map of algebraic sets
$$\mathcal{R}(\Pi,\mathrm{SL}(1,\mathscr{A}))\to \mathbb{F}^N\colon
\rho \mapsto (\chi_\rho(h_1),\ldots,\chi_\rho(h_N))$$
It follows from Lemma \ref{trace_polynomial}
that the value of $\chi_\rho$ at any $g\in\Pi$
is uniquely determined by its values at $h_1,\ldots,h_N$,
implying that the above map is injective on the set of characters,
viewed as functions $\Pi\to\mathbb{F}$.
We denote the image of the above map as 
$$\mathcal{X}(\Pi,\mathrm{SL}(1,\mathscr{A}))\subset \mathbb{F}^N,$$
and decorate $\mathcal{X}$ with superscripts $\mathtt{irr}$ or $\mathtt{red}$ 
for the subsets of irreducible or reducible characters,
respectively.
Hence, we obtain a partition of sets
$$\mathcal{X}(\Pi,\mathrm{SL}(1,\mathscr{A}))=
\mathcal{X}^{\mathtt{irr}}(\Pi,\mathrm{SL}(1,\mathscr{A}))\sqcup 
\mathcal{X}^{\mathtt{red}}(\Pi,\mathrm{SL}(1,\mathscr{A})).$$
We refer to $\mathcal{X}(\Pi,\mathrm{SL}(1,\mathscr{A}))$ 
as the \emph{space of $\mathrm{SL}(1,\mathscr{A})$--characters} for $\Pi$.

In general,
neither $\mathcal{X}(\Pi,\mathrm{SL}(1,\mathscr{A}))$ nor $\mathcal{X}^{\mathtt{red}}(\Pi,\mathrm{SL}(1,\mathscr{A}))$
is an algebraic subset of $\mathbb{F}^N$.
%This is a closed algebraic set in the affine $\mathbb{F}$--linear space 
%$\mathbb{F}^{N}$.
%In fact, $\mathcal{R}(\Pi,\mathrm{SL}(1,\mathscr{A}))$ is 
Instead, 
we can construct an algebraic set $\mathcal{Z}(\Pi,\mathbb{F})$ in $\mathbb{F}^N$
and a Zariski closed subset $\mathcal{Z}^{\mathtt{red}}(\Pi,\mathbb{F})$,
using the $\Integral$--polynomials $P_w$ as guaranteed in Lemma \ref{trace_polynomial}.
The algebraic set 
$$\mathcal{Z}(\Pi,\mathbb{F})\subset \mathbb{F}^N$$
is defined as the locus cut out by the system of equations
$P_r(\chi_\rho(h_1),\ldots,\chi_\rho(h_N))=2$,
ranging over all relators $r=r(X_1,\ldots,X_n)$ for $\Pi$
with respect to the generating tuple.
The Zariski closed subset $\mathcal{Z}^{\mathtt{red}}(\Pi,\mathbb{F})$
is cut out by the extra equations
$P_{[u,v]}(\chi_\rho(h_1),\ldots,\chi_\rho(h_N))=2$,
ranging over all commutators $[u,v]$ of words 
$u=u(X_1,\ldots,X_n)$ and $v=v(X_1,\ldots,X_n)$.
Denote by $\mathcal{Z}^{\mathtt{irr}}(\Pi,\mathbb{F})$
the complement of $\mathcal{Z}^{\mathtt{red}}(\Pi,\mathbb{F})$.
Therefore, we obtain a partition of an algebraic set in $\mathbb{F}^N$
into a Zariski closed subset and a Zariski open subset,
$$\mathcal{Z}(\Pi,\mathbb{F})=
\mathcal{Z}^{\mathtt{irr}}(\Pi,\mathbb{F})\sqcup \mathcal{Z}^{\mathtt{red}}(\Pi,\mathbb{F}).$$
We refer to $\mathcal{Z}(\Pi,\mathbb{F})$ 
as the \emph{space of $\mathbb{F}$--quaternion characters} for $\Pi$.

Every point in $\mathcal{Z}(\Pi,\mathbb{F})$ can be realized
as the character $\chi_\rho$ of some representation $\rho\colon \Pi\to \mathrm{SL}(1,\mathscr{A})$,
for some quaternion algebra $\mathscr{A}$ over $\mathbb{F}$.
In fact, for any algebraically closed field $\Omega$ of characteristic $0$,
any quaternion algebra over $\Omega$ is isomorphic to the matrix algebra $\mathrm{M}_2(\Omega)$.
The characterization $\mathcal{Z}(\Pi,\Omega)=\mathcal{X}^{\mathtt{irr}}(\Pi,\mathrm{SL}(2,\Omega))$
is proved in \cite[Section 1, Corollary 1.4.4]{Culler--Shalen_variety} for $\Omega=\Complex$,
and the proof works verbatim for any algebraically closed field of characteristic $0$.
In general, take $\Omega$ to be an algebraic closure of $\mathbb{F}$.
Given any point in $\mathcal{Z}(\Pi,\mathbb{F})$,
we can realize it as the character of some representation
$\rho'\colon \Pi\to \mathrm{SL}(2,\Omega)$;
if the point lies in $\mathcal{Z}^{\mathtt{red}}(\Pi,\mathbb{F})$,
we can choose $\rho'$ to be abelian.
Then, the image of $\rho'$ is contained in a quaternion algebra $\mathscr{A}'$
over the trace field $\mathbb{F}'$ of $\rho'$.
To be precise,
$\mathbb{F}'$ is equal to the subfield 
$\Rational(\mathrm{tr}(\rho'(h_1)),\ldots,\mathrm{tr}(\rho'(h_N)))$ in $\Omega$,
which is a subfield of $\mathbb{F}$;
for any irreducible $\rho'$,
the subalgebra $\mathbb{F}'[\rho'(g_1),\ldots,\rho'(g_n)]$ 
in $\mathrm{M}_2(\Omega)$ is the unique quaternion algebra $\mathscr{A}'$ over $\mathbb{F}'$
as claimed; if $\rho'$ is abelian, the subalgebra 
$\mathbb{F}'[\rho'(g_1),\ldots,\rho'(g_n)]$ in $\mathrm{M}_2(\Omega)$ 
is actually a field extension over $\mathbb{F}'$ of degree $1$ or $2$,
so it can be further extended to be a quaternion algebra $\mathscr{A}'$ over $\mathbb{F}'$,
as claimed.
Take $\mathscr{A}=\mathscr{A}\otimes_{\mathbb{F}'}\mathbb{F}$
and obtain the scalar extension of $\rho'$
as $\rho\colon \Pi\to \mathrm{SL}(1,\mathscr{A})$,
then $(\mathscr{A},\rho)$ realizes the given point as desired.

For any point in $\mathcal{Z}^{\mathtt{irr}}(\Pi,\mathbb{F})$,
the pairs $(\mathscr{A},\rho)$ to realize the point
are conjugate to each other in $\mathrm{GL}(2,\Omega)$,
where $\Omega$ is an algebraic closure of $\mathbb{F}$.
It follows that the $\mathbb{F}$--algebra isomorphism type of $\mathscr{A}$
is uniquely determined for points in $\mathcal{Z}^{\mathtt{irr}}(\Pi,\mathbb{F})$.
Therefore, we have a partition of sets
$$\mathcal{Z}^{\mathtt{irr}}(\Pi,\mathbb{F})=
\bigsqcup_{[\mathscr{A}/\mathbb{F}]} \mathcal{X}^{\mathtt{irr}}(\Pi,\mathrm{SL}(1,\mathscr{A})),$$
where the disjoint union ranges over all isomorphism classes
of quaternion algebras $\mathscr{A}$ over $\mathbb{F}$.

\begin{example}\label{example_Z_concrete}
For a free group $F_2$ of rank $2$, we can take $N=3$.
\begin{enumerate}
%\item 
%Let $\Pi$ be any finitely generated group.
%The space of complex quaternion characters
%$\mathcal{Z}(\Pi,\Complex)\subset \Complex^N$ is equal to 
%$\mathcal{X}(\Pi,\mathrm{SL}(2,\Complex))$.
%The space of real quaternion characters
%$\mathcal{Z}(\Pi,\Real)$ consists of
%all the real points of $\mathcal{Z}(\Pi,\Complex)$, 
%(that is, the intersection in $\Complex^N$ with $\Real^N$);
%$\mathcal{Z}^{\mathtt{irr}}(\Pi,\Real)$ 
%is the disjoint union of subsets
%$\mathcal{X}^{\mathtt{irr}}(\Pi,\mathrm{SL}(2,\Real))$ and 
%$\mathcal{X}^{\mathtt{irr}}(\Pi,\mathrm{SU}(2))$.
%Both $\mathcal{X}^{\mathtt{irr}}(\Pi,\mathrm{SL}(2,\Real))$ and 
%$\mathcal{X}^{\mathtt{irr}}(\Pi,\mathrm{SU}(2))$ are open subsets of $\mathcal{Z}(\Pi,\Real)$,
%with respect to the (usual) analytic topology of $\Real^N$.
%Moreover, $\mathcal{X}^{\mathtt{irr}}(\Pi,\mathrm{SU}(2))$ 
%is bounded within the $N$-box region $[-2,2]^N$.
\item
Direct computation shows $\mathcal{Z}(F_2,\Real)=\Real^3$.
Moreover, it can be figured out that 
the reducible real quaternion characters of $F_2$ are precisely
the points $(x,y,z)\in\Real^3\colon x^2+y^2+z^2-xyz=4$.
%In fact, the equation is derived from 
%the necessary condition $\chi_\rho([g,h])=2$ of reducibility,
%using the trace relation 
%$\mathrm{Tr}([\alpha,\beta])=
%\mathrm{Tr}(\alpha)^2+\mathrm{Tr}(\beta)^2+\mathrm{Tr}(\alpha\beta)^2
%-\mathrm{Tr}(\alpha)\mathrm{Tr}(\beta)\mathrm(\alpha\beta)-2$.
%On the other hand, every real solution to the equation 
%can be realized by some diagonal representation in 
%$\mathrm{SL}(2,\Real)$ or $\mathrm{SU}(2)$.
The equation defines a real algebraic surface with four singular points,
roughly looking like a tetrahedron with four extra cones emanating from the vertices
opposite to the corners, smoothing out all edges.
It divides $\Real^3$ into six complementary connected components.
Five of them are unbounded, 
all consisting of irreducible $\mathrm{SL}(2,\Real)$--characters.
The only bounded one consists of irreducible $\mathrm{SU}(2)$--characters.
\item
There are infinitely many isomorphically distinct quaternion algebras $A$ over $\Rational$,
which can be completely classified by their places of ramification,
(see \cite[Chapter 7, Theorem 7.3.6]{Maclachlan--Reid_book}).
Every rational point of $\mathcal{Z}^{\mathtt{irr}}(F_2,\Real)$
is the character of some $\mathrm{SL}(1,A)$--representation for some $A$.
If $A\otimes_\Rational\Real\cong \mathrm{M}_2(\Real)$, 
then $\mathrm{SL}(1,A)$ is dense in $\mathrm{SL}(2,\Real)$,
implying that the irreducible $\mathrm{SL}(1,A)$--characters
form a dense subset of rational points in $\mathcal{X}^{\mathtt{irr}}(F_2,\mathrm{SL}(2,\Real))$.
Similarly, if $A\otimes_\Rational\Real\cong \mathcal{H}$,
then the irreducible $\mathrm{SL}(1,A)$--characters
form a dense subset of rational points in $\mathcal{X}^{\mathtt{irr}}(F_2,\mathrm{SU}(2))$.
%In particular, $\mathcal{X}(F_2,\mathrm{SL}(1,A))$ cannot be an algebraic set for any $A$,
%since its Zariski closure stays identical for all $A$.
\end{enumerate}
\end{example}

\subsection{The analytic topology}\label{Subsec_analytic}
We collect some significant facts pertaining to the analytic topology
on the spaces of representations and characters in 
$\mathrm{SL}(2,\Complex)$, and $\mathrm{SL}(2,\Real)$, and $\mathrm{SU}(2)$.

Assume $\Pi$ to be any finitely generated group.
The space of $\mathrm{SL}(2,\Complex)$--characters
$\mathcal{X}(\Pi,\mathrm{SL}(2,\Complex))$ is equal to 
the space of complex characters
$\mathcal{Z}(\Pi,\Complex)\subset\Complex^N$.
This is exactly the spaces of $\mathrm{SL}(2,\Complex)$--characters as introduced by 
Culler and Shalen \cite{Culler--Shalen_variety}.
The intersection of $\mathcal{Z}(\Pi,\Complex)$
with the closed subset $\Real^N$ in $\Complex^N$
can be identified with the space of real quaternions $\mathcal{Z}(\Pi,\Real)$,
which is the union of closed subsets
$\mathcal{X}(\Pi,\mathrm{SL}(2,\Real))$ and $\mathcal{X}(\Pi,\mathrm{SU}(2))$.
The closed subset $\mathcal{X}(\Pi,\mathrm{SU}(2))$ is always contained
in the $N$--dimensional box region $[-2,2]^N$, and hence always compact.
The closed subset $\mathcal{X}(\Pi,\mathrm{SL}(2,\Real))$ is not compact in general.
The open subset $\mathcal{Z}^{\mathtt{irr}}(\Pi,\Real)$ of $\mathcal{Z}(\Pi,\Real)$
is the topological disjoint union of open subsets
$\mathcal{X}^{\mathtt{irr}}(\Pi,\mathrm{SL}(2,\Real))$ and 
$\mathcal{X}^{\mathtt{irr}}(\Pi,\mathrm{SU}(2))$.

In the following facts, $G$ refers to any linear group among
$\mathrm{SL}(2,\Real)$, $\mathrm{SU}(2)$, and $\mathrm{SL}(2,\Complex)$.

The topological space $\mathcal{X}(\Pi,G)$ has finitely many connected components.
This follows from the fact that $\mathcal{X}(\Pi,G)$ is 
the image of the real algebraic set $\mathcal{R}(\Pi,G)$ under a continuous map, 
and a theorem due to Whitney regarding real algebraic sets \cite[Theorem 3]{Whitney_real}.
The finiteness of connected components also holds
for $\mathcal{X}^{\mathtt{irr}}(\Pi,G)$ (see \cite[Theorem 4]{Whitney_real}).

The topological space pair 
$(\mathcal{X}(\Pi,G),\mathcal{X}^{\mathtt{red}}(\Pi,G))$
is triangulable.
Namely, it is homeomorphic to 
a pair of a locally finite simplicial complex and a subcomplex.
In fact, the pairs $(\mathcal{Z}(\Pi,\Real),\mathcal{Z}^{\mathtt{red}}(\Pi,\Real))$
and $(\mathcal{Z}(\Pi,\Complex),\mathcal{Z}^{\mathtt{red}}(\Pi,\Complex))$
are both triangulable, by a theorem due to Hardt regarding analytic sets \cite[Theorem 3]{Hardt_triangulation}.
The sub-pair $(\mathcal{X}(\Pi,G),\mathcal{X}^{\mathtt{red}}(\Pi,G))$ inherits a triangulation
either from $(\mathcal{Z}(\Pi,\Real),\mathcal{Z}^{\mathtt{red}}(\Pi,\Real))$, if $G$ is $\mathrm{SL}(2,\Real)$ or $\mathrm{SU}(2)$,
or from $(\mathcal{Z}(\Pi,\Complex),\mathcal{Z}^{\mathtt{red}}(\Pi,\Complex))$, if $G$ is $\mathrm{SL}(2,\Complex)$.

The natural continuous map $\mathcal{R}(\Pi,G)\to \mathcal{X}(\Pi,G)\colon \rho\mapsto\chi_\rho$
satisfies the weak path lifting property.
Namely, every continuous path $[0,1]\to \mathcal{X}(\Pi,G)$
lifts to some path $[0,1]\to \mathcal{R}(\Pi,G)$.
This is a special case of a more general fact 
regarding real reductive Lie groups, 
as pointed out by Biswas, Lawton, and Ramras \cite[Lemma 2.1]{BLR_character}.
Moreover, restricted to the irreducibles,
$\mathcal{R}^{\mathtt{irr}}(\Pi,G)\to \mathcal{X}^{\mathtt{irr}}(\Pi,G)$
is a principal bundle projection with structure group $G/\{\pm\mathbf{1}\}$,
so the strong path lifting property holds in this case.
Namely, 
any base point lift $\{0\}\to \mathcal{R}^{\mathtt{irr}}(\Pi,G)$
of any continuous path $[0,1]\to \mathcal{X}^{\mathtt{irr}}(\Pi,G)$
extends to some path lift $[0,1]\to \mathcal{R}^{\mathtt{irr}}(\Pi,G)$.

%\begin{example}\label{example_Z_fg}
%Let $\Pi$ be any finitely generated group.
%The space of complex quaternion characters
%$\mathcal{Z}(\Pi,\Complex)\subset \Complex^N$ is equal to 
%$\mathcal{X}(\Pi,\mathrm{SL}(2,\Complex))$.
%The space of real quaternion characters
%$\mathcal{Z}(\Pi,\Real)$ consists of
%all the real points of $\mathcal{Z}(\Pi,\Complex)$, 
%(that is, the intersection in $\Complex^N$ with $\Real^N$);
%$\mathcal{Z}^{\mathtt{irr}}(\Pi,\Real)$ 
%is the disjoint union of subsets
%$\mathcal{X}^{\mathtt{irr}}(\Pi,\mathrm{SL}(2,\Real))$ and 
%$\mathcal{X}^{\mathtt{irr}}(\Pi,\mathrm{SU}(2))$.
%Both $\mathcal{X}^{\mathtt{irr}}(\Pi,\mathrm{SL}(2,\Real))$ and 
%$\mathcal{X}^{\mathtt{irr}}(\Pi,\mathrm{SU}(2))$ are open subsets of $\mathcal{Z}(\Pi,\Real)$,
%with respect to the (usual) analytic topology of $\Real^N$.
%Moreover, $\mathcal{X}^{\mathtt{irr}}(\Pi,\mathrm{SU}(2))$ 
%is always contained in the $N$--box region $[-2,2]^N$.
%\end{example}

\section{Trace slices}\label{Sec-trace_slices}
In this section, we introduce trace slices
of types $\mathrm{SL}(2,\Complex)$, $\mathrm{SL}(2,\Real)$, and $\mathrm{SU}(2)$.
We analyze their properties from the unified perspective of quaternion algebras,
generalizing prototypical results in the $\mathrm{SU}(2)$ case \cite{Heusener--Kroll_abelian,Lin_invariant}.
We also consider parametric trace slices, allowing varying trace value.
In the proof of Theorem \ref{main_deform} in Section \ref{Sec-deform},
the parametric version are important 
for analyzing change of 
the local modulo $2$ intersection numbers (see Lemma \ref{half_signature_parity}).

\begin{notation}\label{notation_slice}
	Denote 
	${WF}_n=\langle x_1,\ldots,x_n,y_1,\ldots,y_n\,|\,x_1\cdots x_n=y_1\cdots y_n\rangle$,
	and
	$F_n=\langle x_1,\ldots,x_n\rangle$,
	as presented groups.
	\begin{enumerate}
	\item
	For any $\tau\in\Complex$ if $G$ is $\mathrm{SL}(2,\Complex)$, 
	or for any $\tau\in\Real$ if $G$ is	$\mathrm{SL}(2,\Real)$ or $\mathrm{SU}(2)$,
	the \emph{trace $\tau$--slice of $G$--representations} for $WF_n$ refers to the set
	$$\mathcal{R}_\tau({WF}_n,G)
	=
	\left\{
	\rho\in\mathcal{R}({WF}_n,G)
	\middle| 
	\begin{array}{l}
	\mathrm{tr}(\rho(x_\mu))=\mathrm{tr}(\rho(y_\mu))=\tau,\\
	%\mbox{for all }\mu\in\{1,\ldots,n\}\\
	\mbox{for all }\mu=1,\ldots,n
	\end{array}
	\right\}.$$
	where $\mathrm{tr}$ denotes the usual trace of a square matrix.	
	The \emph{trace $\tau$--slice of $G$--characters},
	denoted as $\mathcal{X}_\tau(WF_n,G)$,
	the image of $\mathcal{R}_\tau(WF_n,G)$ in the space of $G$--characters 
	under the natural map $\rho\mapsto\chi_\rho$.
	We decorate $\mathcal{R}_\tau$ or $\mathcal{X}_\tau$ with superscripts 
	$\mathtt{irr}$ or $\mathtt{red}$ 
	to denote the subsets of irreducible or reducible
	representations or characters.
	We view these spaces as equipped with the analytic topology.	
	\item
	More generally,
	for any subset $T\subset \Complex$ or $T\subset \Real$, accordingly, 
	we introduce the \emph{parametric trace $T$--slice of $G$--representations} as
	$$
	\mathcal{R}_{T}(WF_n,G) 
	=\bigcup_{\tau\in T} \mathcal{R}_\tau(WF_n,G),$$
	viewed as a topological subspace of $\mathcal{R}(WF_n,G)$.
	The \emph{parametrization map}
	refers to the continuous map 
	$$\mathcal{R}_T(WF_n,G)\to T,$$
	such that each trace $\tau$--slice projects to the trace value $\tau$.
	Similarly, we introduce $\mathcal{X}_{T}(WF_n,G)$ and 
	the decorated notations with superscripts $\mathtt{irr}$ or $\mathtt{red}$.	
	\item	
	For any $(G,\tau)$ as above, we introduce	
	$$\mathcal{R}_\tau(F_n,G)
	=
	\left\{\rho\in\mathcal{R}(F_n,G)
	\middle| 
	%\mathrm{tr}(\rho(x_1))=\cdots=\mathrm{tr}(\rho(x_n))=\tau
	\begin{array}{l}
	\mathrm{tr}(\rho(x_\mu))=\tau,\\
	%\mbox{for all }\mu\in\{1,\ldots,n\}
	\mbox{for all }\mu=1,\ldots,n
	\end{array}
	\right\},$$
	Similarly,
	we introduce 
	$\mathcal{X}_\tau(F_n,G)$, and the decorated notations for $F_n$.
	For any $(G,T)$ as above,
	we also introduce the parametric versions of trace slices for $F_n$.
	\end{enumerate}
\end{notation}

\begin{lemma}\label{slice_regularity_F}
	Adopt Notation \ref{notation_slice}.
	\begin{enumerate}
	\item
	For $G$ being $\mathrm{SL}(2,\Complex)$, and for any complex $\tau\neq\pm2$,
	$\mathcal{X}^{\mathtt{irr}}_\tau(F_n,G)$ is a smooth manifold
	of dimension $4n-6$.
	\item
	For $G$ being $\mathrm{SL}(2,\Real)$ or $\mathrm{SU}(2)$, and for any real $\tau\neq\pm2$,
	$\mathcal{X}^{\mathtt{irr}}_\tau(F_n,G)$ is a smooth manifold
	of dimension $2n-3$.
	\end{enumerate}
\end{lemma}

\begin{proof}
	Denote by $\delta_G$ the dimension of any real Lie group $G$.
	Namely, 
	$\delta_G=6$ if $G$ is $\mathrm{SL}(2,\Complex)$, 
	or $\delta_G=3$ if $G$ is $\mathrm{SL}(2,\Real)$ or $\mathrm{SU}(2)$.
	Denote by 
	$$G_\tau=\{A\in G\colon\mathrm{tr}(A)=\tau\}$$
	the trace $\tau$--slice of $G$, 
	which is invariant under the conjugation action of $G$ 
	and under the inverse.
	For any $\tau\neq\pm2$,
	$G_\tau$ comprises either a unique conjugacy class
	(in $\mathrm{SL}(2,\Complex)$, or in $\mathrm{SL}(2,\Real)$ of hyperbolic type, or in $\mathrm{SU}(2)$),
	or an inverse pair of conjugacy classes 
	(in $\mathrm{SL}(2,\Real)$ of elliptic type);
	the stabilizer of any $A\in G_\tau$ is the centralizer $Z_G(A)$ of $A$ in $G$,
	which is a closed Lie subgroup of dimension $\delta_G/3$.
	In particular,
	$G_\tau$ is	a properly embedded smooth submanifold in $G$ of dimension $2\delta_G/3$.	
	
	Since $F_n$ is freely generated by $x_1,\ldots,x_n$,
	$\mathcal{R}_\tau(F_n,G)$ 
	can be identified obviously with the $n$--fold cartesian product $G_\tau^n$.
	Restricted to the irreducibles, the natural map 
	$\mathcal{R}^{\mathtt{irr}}_\tau(F_n,G)\to\mathcal{X}^{\mathtt{irr}}_\tau(F_n,G)$
	is a bundle projection with fibers diffeomorphic to $G/\{\pm \mathbf{1}\}$.
	It follows that $\mathcal{X}^{\mathtt{irr}}_\tau(F_n,G)$ is a smooth manifold
	of dimension $n\cdot2\delta_G/3-\delta_G=(2n-3)\cdot\delta_G/3$.
\end{proof}

%\begin{lemma}\label{slice_regularity_WF}
	%Adopt Notation \ref{notation_slice}.
	%Denote by $\delta_G$ the dimension of $G$ as a real Lie group.
	%
	%If $\tau\neq\pm2$,
	%then the trace $\tau$--slice of irreducible $G$--characters
	%$\mathcal{X}^{\mathtt{irr}}_\tau({WF}_n,G)$ 
	%is a smooth manifold of dimension $(4n-6)\cdot\delta_G/3$.
%\end{lemma}

\begin{lemma}\label{slice_regularity_WF}
	Adopt Notation \ref{notation_slice}.
	\begin{enumerate}
	\item
	For $G$ being $\mathrm{SL}(2,\Complex)$, and for any complex $\tau\neq\pm2$,
	$\mathcal{X}^{\mathtt{irr}}_\tau({WF}_n,G)$ is a smooth manifold
	of dimension $8n-12$.
	\item
	For $G$ being $\mathrm{SL}(2,\Real)$ or $\mathrm{SU}(2)$, and for any real $\tau\neq\pm2$,
	$\mathcal{X}^{\mathtt{irr}}_\tau(WF_n,G)$ is a smooth manifold
	of dimension $4n-6$.
	\end{enumerate}
\end{lemma}

\begin{proof}
	Denote $G_\tau=\{A\in G\colon\mathrm{tr}(A)=\tau\}$.
	As explained in the proof of Lemma \ref{slice_regularity_F},
	for any $\tau\neq\pm2$,
	$G_\tau$ is a smooth manifold of dimension $2\delta_G/3$,
	where $\delta_G$ denotes the real dimension of $G$.
	
	To understand the tangent space $T_AG_\tau$ at any $A\in G_\tau$,
	we identify $T_{\mathbf{1}}G$ with the Lie algebra $\mathfrak{g}$ of $G$,
	viewed as a real linear subspace of $\mathrm{M}_2(\Complex)$ (see Example \ref{q_algebra_example}).
	%More specifically, $\mathfrak{sl}(2,\Complex)$ 
	%consisting of all the complex square matrices of size $2$ with trace $0$,
	%and $\mathfrak{sl}(2,\Real)$ consists of all the real matrices in $\mathfrak{sl}(2,\Complex)$,
	%and	$\mathfrak{su}(2)$ consists of all the skew-Hermitian matrices in $\mathfrak{sl}(2,\Complex)$.
	The tangent space $T_A G$ at any $A\in G$ is identified with $A\mathfrak{g}$,
	consisting of all $AX\in\mathrm{M}_2(\Complex)$ for all $X\in\mathfrak{g}$.
	For any $\tau\neq\pm2$ and any $A\in G_\tau$,
	a tangent vector $AX\in T_A G_\tau$ 
	is characterized by the linear equation $\mathrm{tr}(AX)=0$.
	Since $X\in\mathfrak{g}$ implies $\mathrm{tr}(X)=0$,
	$\mathrm{tr}(AX)=0$ is equivalent to $\mathrm{tr}((2A-\tau\mathbf{1})X)=0$.
	On the other hand, 
	using the characteristic equation $A^2-\mathrm{tr}(A)\cdot A+\mathbf{1}=0$,
	we observe that $A-A^{-1}=2A-\tau\mathbf{1}$	lies in $\mathfrak{g}$.
	In other words, for any $\tau\neq\pm2$ and for any $A\in G_\tau$,
	$AX\in T_AG_\tau$ is equivalent to $X\in (2A-\tau\mathbf{1})^\perp$,
	where $\perp$ denotes the orthogonal complement in $\mathfrak{g}$,
	with respect to 
	the (nondegenerate, complex or real symmetric) trace pairing $(X,Y)\mapsto \mathrm{tr}(XY)$.
	For any $B\in G$, the adjoint action $X\mapsto BXB^{-1}$ of $B$ on $\mathfrak{g}$
	preserves $(2A-\tau\mathbf{1})^\perp$ if and only if 
	it preserves $2A-\tau\mathbf{1}$.
	The last condition is equivalent to the equation $BAB^{-1}=A$ in $G$.
	
	We summarize the above discussion as follows.
	For any $\tau\neq\pm2$ and for any $A\in G_\tau$,
	we can identify
	$$T_AG_\tau=AH_A^\perp$$
	as a real linear subspace of $T_AG=A\mathfrak{g}$,
	where $H_A=2A-\tau\mathbf{1}$ lies in $\mathfrak{g}$.
	Moreover,
	for any $B\in G$,
	$BH_A^\perp B^{-1}=H_A^\perp$ holds if and only if $B$ commutes with $A$. 
	
	Since ${WF}_n$ is generated by $x_1,\ldots,x_n,y_1,\ldots,y_n$,
	subject to $x_1\cdots x_n=y_1\cdots y_n$,
	$\mathcal{R}_\tau({WF}_n,G)$ 
	can be identified with the preimage of $\mathbf{1}\in G$
	for the map 
	$$G_\tau^{2n}\to G\colon
	(A_1,\ldots,A_{2n}) \mapsto 	A_1\cdots A_{2n},$$
	observing $\mathrm{tr}(B)=\mathrm{tr}(B^{-1})$ for any $B\in G$.
	Assuming $(A_1,\ldots,A_{2n})\in G_\tau^{2n}$ and $A_1\cdots A_{2n}=\mathbf{1}$,
	the tangent map at $(A_1,\ldots,A_{2n})$
	is a real linear homomorphism
	$T_{A_1}G_\tau\oplus\cdots\oplus T_{A_{2n}}G_\tau\to T_{\mathbf{1}} G$,
	which is identified as
	$$A_1H_{A_1}^\perp\oplus\cdots \oplus A_{2n}H_{A_{2n}}^\perp \to \mathfrak{g}.$$
	By direct computation, this tangent map takes the form
	$$(A_1X_1,\ldots,A_{2n}X_{2n}) \mapsto \sum_\nu S_\nu X_\nu S_\nu^{-1},$$
	%\begin{eqnarray*}
	%A_1H_{A_1}^\perp\oplus\cdots A_{2n}H_{A_{2n}}^\perp &\longrightarrow& \mathfrak{g}\\
	%(A_1X_1,\ldots,A_{2n}X_{2n}) &\mapsto& \sum_\nu S_\nu X_\nu S_\nu^{-1}
	%\end{eqnarray*}
	summing over all $\nu=1,\ldots,2n$, 
	where we denote $S_\nu=A_1\cdots A_\nu$ in $G$.
	Note that the tangent map restricted to 
	each direct summand $A_\nu H_{A_\nu}^\perp$
	has rank $2$,
	over $\Complex$ if $\mathfrak{g}$ is $\mathfrak{sl}(2,\Complex)$,
	or over $\Real$ if $\mathfrak{g}$ is $\mathfrak{sl}(2,\Real)$ or $\mathfrak{su}(2)$.
	If this tangent map is not surjective,
	then we must have $S_\nu H_{A_\nu}^\perp S_\nu^{-1}$ identical for all $\nu=1,\ldots,2n$,
	implying that $A_1,\ldots,A_{2n}$ all commute with each other.
	This cannot be the case if 
	$(A_1,\ldots,A_{2n})$
	is an irreducible $G$--representation for ${WF}_n$.
	
	The above argument implies that for any $\tau\neq\pm2$,
	the map $G^{2n}_\tau\to G$ is a local submersion onto a neighborhood of $\mathbf{1}$
	at any preimage point in $\mathcal{R}^{\mathtt{irr}}_\tau({WF}_n,G)$.
	It follows that $\mathcal{R}^{\mathtt{irr}}_\tau({WF}_n,G)$ is a smooth manifold
	of dimension $4n\delta_G/3-\delta_G=(4n-3)\cdot\delta_G/3$.
	As $\mathcal{R}^{\mathtt{irr}}_\tau({WF}_n,G)\to\mathcal{X}^{\mathtt{irr}}_\tau({WF}_n,G)$
	is a bundle projection with fibers diffeomorphic to $G/\{\pm \mathbf{1}\}$,
	it follows that $\mathcal{X}^{\mathtt{irr}}_\tau({WF}_n,G)$ is a smooth manifold
	of dimension $(4n-3)\cdot\delta_G/3-\delta_G=(4n-6)\cdot\delta_G/3$.
\end{proof}

\begin{lemma}\label{reducible_finiteness}
	Adopt Notation \ref{notation_slice}.
	For any complex $\tau\neq\pm2$, if $G$ is $\mathrm{SL}(2,\Complex)$,
	or for any real $\tau\neq\pm2$, if $G$ is $\mathrm{SL}(2,\Real)$ or $\mathrm{SU}(2)$,
	$\mathcal{X}^{\mathtt{red}}_\tau(F_n,G)$ is a finite set of cardinality $2^{n-1}$,
	and $\mathcal{X}^{\mathtt{red}}_\tau({WF}_n,G)$ is a finite set of cardinality $\leq 2^{2n-1}$.
\end{lemma}

\begin{proof}
	Fix some $A\in G$ with trace $\mathrm{tr}(A)=\tau$. 
	If $\tau\neq\pm2$, then $A\neq A^{-1}$.
	For the case with $F_n=\langle x_1,\ldots,x_n\rangle$, 
	any $G$--representation of $F_n$ with reducible $G$--character in the trace $\tau$--slice
	takes the form $x_1\mapsto A^{\epsilon_1},\ldots,x_n\mapsto A^{\epsilon_n}$, up to conjugacy,
	for some tuple $(\epsilon_1,\ldots,\epsilon_n)\in \{\pm1\}^n$.
	A pair of such tuples yield identical $G$--character 
	if and only if they differ by an overall sign flip.
	This shows that $\mathcal{X}^{\mathtt{red}}_\tau(F_n,G)$ is a finite set of cardinality $2^{n-1}$.
	Similarly, 
	the number of reducible $G$--characters of ${WF}_n$ in the trace $\tau$--slice is at most $2^{2n-1}$.
	Indeed, as one may easily check, the cardinality is equal to $\frac{1}{2}{{2n}\choose{n}}$,
	unless $A$ has finite even order $\leq2n$.
\end{proof}

\begin{lemma}\label{parametric_slices}
	Adopt Notation \ref{notation_slice}.
	\begin{enumerate}
	\item 
	For $G$ being $\mathrm{SL}(2,\Complex)$,
	and for any open subset $T\subset \Complex\setminus\{\pm2\}$,
	$\mathcal{X}^{\mathtt{irr}}_T(F_n,G)$ is a smooth manifold of dimension $4n-4$,
	and
	$\mathcal{X}^{\mathtt{irr}}_T({WF}_n,G)$ is a smooth manifold of dimension $8n-10$.
	\item
	For $G$ being $\mathrm{SL}(2,\Real)$ or $\mathrm{SU}(2)$,
	and for any open subset $T\subset \Real\setminus\{\pm2\}$,
	$\mathcal{X}^{\mathtt{irr}}_T(F_n,G)$ is a smooth manifold of dimension $2n-2$,
	and
	$\mathcal{X}^{\mathtt{irr}}_T({WF}_n,G)$ is a smooth manifold of dimension $4n-5$.
	\item 
	For any $(G,T)$ as above, the parametrization maps
	$\mathcal{X}^{\mathtt{irr}}_T(F_n,G)\to T$ and
	$\mathcal{X}^{\mathtt{irr}}_T({WF}_n,G)\to T$ are smooth submersions,
	and indeed, bundle projections over each connected component of $T$.
	\item
	For any $(G,T)$ as above, the parametrization map
	$\mathcal{X}^{\mathtt{red}}_T(F_n,G)\to T$ is a finite covering projection,
	$\mathcal{X}^{\mathtt{red}}_T({WF}_n,G)\to T$ is a finite covering projection 
	restricted away from an isolated finite subset.
	\end{enumerate}
\end{lemma}

\begin{proof}
	We only illustrate by showing the assertions regarding 
	$\mathcal{X}^{\mathtt{irr}}_T({WF}_n,G)$ and $\mathcal{X}^{\mathtt{irr}}_T({WF}_n,G)\to T$.
	
	Since $T$ is an open subset of $\Complex\setminus\{\pm2\}$ or $\Real\setminus\{\pm2\}$,
	we obtain $G_T=\bigcup_{\tau\in T} G_\tau$ as a smooth submanifold in $G$,
	and $(G^{2n})_T=\bigcup_{\tau\in T}G^{2n}_\tau$ as a smooth submanifold in $G^{2n}$,
	both submersing $T$ via the parametrization map.
	These submersions are bundle projections over each connected component of $T$.
	Using the maps $G^{2n}_\tau\to G$ parametrized by $\tau\in T$ (as in the proof of Lemma \ref{slice_regularity_WF})
	and the parametrization map $(G^{2n})_T\to T$,
	we obtain a smooth map $(G^{2n})_T\to G\times T$.
	At any point $(A_1,\ldots,A_n)_\tau\in (G^{2n})_\tau$,
	such that $A_1\cdots A_{2n}=\mathbf{1}$ holds in $G$,
	and such that $\langle A_1,\ldots,A_{2n}\rangle$ is irreducible in $G$,
	the tangent map of $(G^{2n})_T\to G\times T$ is obviously surjective.
	In fact, at any such point,
	the subspace tangent to $G^{2n}_\tau$ surjects the subspace tangent to the factor $G$,
	and the induced maps between the quotient spaces can be identified as the identity map
	between the tangent space of $T$ at $\tau$.
	It follows that $\mathcal{R}^{\mathtt{irr}}_T({WF}_n,G)$ is a smooth manifold, 
	submersing $T$ as a bundle projection over each connected component of $T$.
	Moreover, the dimension of $\mathcal{R}^{\mathtt{irr}}_T({WF}_n,G)$
	is equal to $(4n-3)\cdot\delta_G/3+\delta_G/3=(4n-2)\cdot\delta_G/3$,
	where $\delta_G$ denotes the real dimension of $G$.
	The character map $\mathcal{R}^{\mathtt{irr}}_T({WF}_n,G)\to\mathcal{X}^{\mathtt{irr}}_T({WF}_n,G)$
	is a bundle projection with fibers diffeomorphic to $G/\{\pm\mathbf{1}\}$,
	and commutes with the parametrization maps onto $T$.
	It follows that $\mathcal{X}^{\mathtt{irr}}_T({WF}_n,G)$ has dimension 
	$(4n-2)\cdot\delta_G/3-\delta_G=(4n-5)\cdot\delta_G/3$,
	and the parametrization map $\mathcal{X}^{\mathtt{irr}}_T({WF}_n,G)\to T$ 
	is a submersion,
	and moreover, a bundle projection onto each connected component of $T$.
	
	Similarly, the other assertions follow from the proofs of 
	Lemmas \ref{slice_regularity_F} and \ref{reducible_finiteness},
	by arguing in the parametric way.
\end{proof}

\section{Equivariant transversality}\label{Sec-transversality}
In this section, we introduce what we call 
$(\nu,J)$--structures and $(\nu,J)$--manifolds (Definitions \ref{nu_J_structure_def} and \ref{nu_J_manifold_def}).
We develop a minimal toolkit of differential topology for disposing $(\nu,J)$--equivariant transversality 
(Lemmas \ref{nu_J_manifold_property}, \ref{nu_J_transverse_property}, and \ref{nu_J_transverse_approximation}),
which suffices for our subsequent application.
Transversality in the $(\nu,J)$--equivariant setting
appears more closely analogous to the usual theory in differential topology of smooth manifolds,
compared to other general equivariant settings in the literature (see Remark \ref{equivariant_transversality_literature}).
%All manifolds in this section are assumed to be second countable.

For standard techniques in differential topology of smooth manifolds, 
we refer to Hirsch's textbook \cite{Hirsch_book}.

\begin{definition}\label{nu_J_structure_def}
For any smooth manifold $Q$ (without boundary and of designated dimension),
a \emph{$(\nu,J)$--structure} on $Q$ refers to any pair $(\nu,J)$ as follows.
\begin{itemize}
\item The item $\nu\colon Q\to Q$ is a smooth self-diffeomorphism on $Q$.
We require $\nu$ to be an involution (namely, $\nu^2=\mathrm{id}$),
and denote by $Q^\nu$ the fixed locus of $\nu$.
\item The item $J\colon TQ|_{Q^\nu}\to TQ|_{Q^\nu}$ is 
a smooth vector bundle automorphism on the tangent bundle $TQ\to Q$ restricted over $Q^\nu$.
We require $J$ to be an almost complex structure (namely, $J^2=-\mathbf{1}$ fiberwise).
\item Moreover, we require the tangent map $\nu_*|_{P^\nu}\colon TQ|_{Q^\nu}\to TQ|_{Q^\nu}$
restricted over $Q^\nu$ to be $J$--antilinear (namely, $\nu_*\circ J=-J\circ \nu_*$ fiberwise).
\end{itemize}
\end{definition}

\begin{definition}\label{nu_J_manifold_def}
A smooth \emph{$(\nu,J)$--manifold} 
refers to a smooth manifold $Q$ furnished with a $(\nu,J)$--structure $(\nu,J)=(\nu_Q,J_Q)$.
A closed or open smooth \emph{$(\nu,J)$--submanifold}, respectively,
refers to a closed or open smooth submanifold $S$ of a smooth $(\nu,J)$--manifold $Q$,
such that $S$ is $\nu$--invariant, and such that $TS|_{S^\nu}$ is $J$--invariant.
A \emph{$(\nu,J)$--equivariant} smooth map
refers to a smooth map $f\colon P\to Q$
between smooth $(\nu,J)$--manifolds $P$ and $Q$,
such that $f$ is $\nu$--equivariant (namely, $f\circ\nu=\nu\circ f$),
and such that $f_*|_{P^\nu}\colon TP|_{P^\nu}\to TQ|_{Q^\nu}$ is $J$--linear 
(namely, $f_*\circ J=J\circ f_*$ fiberwise).
A \emph{$(\nu,J)$--equivariant} smooth homotopy
refers to a smooth homotopy $P\times[0,1]\to Q\colon (x,t)\mapsto f_t(x)$
between a pair of $(\nu,J)$--equivariant smooth maps $f_0,f_1\colon P\to Q$
(constant in $t$ nearby $P\times\{0,1\}$),
such that $f_t\colon P\to Q$ is $(\nu,J)$--equivariant for all $t\in[0,1]$.
\end{definition}

\begin{example}\label{nu_J_example}
	Let $T\subset \Complex\setminus\{\pm2\}$ be any complex-conjugation invariant, open subset.
	Adopt Notations \ref{notation_braid} and \ref{notation_slice}.
	\begin{enumerate}
	\item
	The parametric trace $T$--slice $\mathcal{X}^{\mathtt{irr}}_T({WF}_n;\mathrm{SL}(2,\Complex))$
	of irreducible $\mathrm{SL}(2,\Complex)$--characters 
	is equipped with a natural $(\nu,J)$--structure, as follows.
	The involution $\nu$ on $\mathcal{X}^{\mathtt{irr}}_T({WF}_n;\mathrm{SL}(2,\Complex))$
	is induced by the complex conjugation of $\mathrm{SL}(2,\Complex)$--characters
	(that is, taking the complex conjugation of the values for any character).
	The fixed locus of $\nu$ is exactly 
	$\mathcal{X}^{\mathtt{irr}}_T({WF}_n;\mathrm{SL}(2,\Real))
	\sqcup \mathcal{X}^{\mathtt{irr}}_T({WF}_n;\mathrm{SU}(2))$.
	Note that $\mathcal{X}^{\mathtt{irr}}_T({WF}_n;\mathrm{SL}(2,\Complex))$ is 
	a Zariski open subset of a complex algebraic set (Subsection \ref{Subsec_analytic}), 
	and is a smooth manifold (Lemma \ref{parametric_slices}),
	so it is also an analytic complex manifold.
	The almost complex structure $J$ along the fixed locus of $\nu$ is induced
	by the complex manifold structure of $\mathcal{X}^{\mathtt{irr}}_T(F_n;\mathrm{SL}(2,\Complex))$.
	Similarly, $\mathcal{X}^{\mathtt{irr}}_T(F_n;\mathrm{SL}(2,\Complex))$ is also 
	a smooth $(\nu,J)$--manifold.
	\item
	The parametrization map
	$\mathcal{X}^{\mathtt{irr}}_T({WF}_n;\mathrm{SL}(2,\Complex))\to T$
	is a $(\nu,J)$--equivariant smooth bundle projection,
	viewing $T$ as an open $(\nu,J)$--submanifold of $\Complex$ 
	(Lemma \ref{parametric_slices}).
	Here, 
	we may apply Lemma \ref{parametric_slices} first to $\Complex\setminus\{\pm2\}$
	and then restrict over $T$.
	Hence, for any $\tau\in T$, the trace $\tau$--slice
	$\mathcal{X}^{\mathtt{irr}}_\tau({WF}_n;\mathrm{SL}(2,\Complex))\to T$
	is a closed smooth $(\nu,J)$--submanifold of 
	$\mathcal{X}^{\mathtt{irr}}_T({WF}_n;\mathrm{SL}(2,\Complex))$.
	The similar statements hold with $F_n$ in place of ${WF}_n$.
	\item
	For any $\sigma\in\mathscr{B}_n$,
	the pullback map between 
	parametric trace $T$--slice of irreducible $\mathrm{SL}(2,\Complex)$--characters 
	$r^*_\sigma\colon\mathcal{X}^{\mathtt{irr}}_T(F_n;\mathrm{SL}(2,\Complex))
	\to \mathcal{X}^{\mathtt{irr}}_T({WF}_n;\mathrm{SL}(2,\Complex))$,
	as naturally induced 
	by the group homomorphism $r_\sigma\colon {WF}_n\to F_n$,
	is a $(\nu,J)$--equivariant smooth proper embedding.
	The image of $r^*_\sigma$ is a closed smooth $(\nu,J)$--submanifold of 
	$\mathcal{X}^{\mathtt{irr}}_T({WF}_n;\mathrm{SL}(2,\Complex))$.
	Here, the closedness and the properness follow from the fact
	that the $\mathrm{SL}(2,\Complex)$--characters form an algebraic set,
	containing the reducibles as a Zariski closed subset (Subsection \ref{Subsec_analytic}).
	Note also that $r^*_\sigma$ is a regular map between the parametric trace $T$--slices of
	$\mathrm{SL}(2,\Complex)$--characters, preserving reducibility and irreducibility.
	The injectivity of $r^*_\sigma$ follows from the surjectivity of $r_\sigma$.
	In particular, these statements hold for $r^*_{\mathrm{id}}$.
	\end{enumerate}
\end{example}

\begin{remark}\label{nu_J_orientation_remark}
	We mention that
	all the $(\nu,J)$--manifolds in Example \ref{nu_J_example} are canonically oriented.
	Moreover, there are canonical orientations of their $\nu$--fixed loci,
	and the $J$--induced orientations near the $\nu$--fixed loci 
	agree with the canonical orientation of the $(\nu,J)$--manifold.
	In fact, these canonical orientations can be defined
	according to the proofs of Lemmas \ref{slice_regularity_F} and \ref{slice_regularity_WF},
	(as the induced orientations of the submersion preimages therein).
	However, these orientations are not so important for our application,
	because we only need to count intersections modulo $2$ in this paper.
	For the same reason, 
	we do not discuss orientations for $(\nu,J)$--manifolds in this section.
	One may come up with several possible variants of this notion,
	depending on the orientability of the manifold and that of the fixed locus.
\end{remark}

\begin{lemma}\label{nu_J_manifold_property}
	Let $Q$ be a smooth $(\nu,J)$--manifold of dimension $q$.
	Then, $q$ is an even number, and $Q^\nu$ is a closed smooth submanifold in $P$ of dimension $q/2$.
	Moreover, $TQ|_{Q^\nu}$ decomposes as the direct sum of vector subbundles $TQ^\nu$ and $J(TQ^\nu)$,
	on which $\nu_*|_{Q^\nu}$ acts as scalar $+1$ and $-1$, respectively.
\end{lemma}

\begin{proof}
	Since $\nu$ is a smoothly diffeomorphic involution on $Q$,
	the fixed locus $Q^\nu$ is a closed smooth submanifold of $Q$.
	The subbundle $TQ^\nu$ is precisely the fixed subbundle of $TQ|_{Q^\nu}$ under $\nu_*|_{Q^\nu}$.
	Since $\nu_*|_{Q^\nu}$ is $J$--antilinear, $J(TQ^\nu)$ is invariant under $\nu_*|_{Q^\nu}$
	and $\nu_*|_{Q^\nu}$ acts as scalar $-1$ on $J(TQ^\nu)$. 
	So, the eigenvector subbundle decomposition of $TQ|_{Q^\nu}$
	with respect to $\nu_*|_{Q^\nu}$ is precisely $TQ^{\nu}\oplus J(TQ^{\nu})$.
	Hence, the dimension $q$ of $Q$ is even, and $Q^\nu$ is of dimension $q/2$.	
\end{proof}
%
%\begin{definition}\label{nu_J_equivariant_def}
%\end{definition}

\begin{lemma}\label{nu_J_transverse_property}
	Let $f\colon P\to Q$ is a $(\nu,J)$--equivariant smooth map 
	between smooth $(\nu,J)$--manifolds $P$ and $Q$.
	Let $S$ is a closed smooth $(\nu,J)$--submsnifold in $Q$ of codimension $k$.
	
	Suppose that $f$ is transverse to $S$ on some $\nu$--invariant open subset $U$ of $P$.
	Then, the restricted map $f^\nu\colon P^\nu\to Q^\nu$ is transverse to $S^\nu$
	on the open subset $U^\nu=U\cap P^\nu$ of $P^\nu$.
	Moreover, $f^{-1}(S)\cap U$ is a closed smooth $(\nu,J)$--submanifold in $U$	of codimension $k$.
\end{lemma}

\begin{proof}
	For any point $x\in U$, 
	the transversality assumption means $f_*T_xP+T_{f(x)}S=T_{f(x)}Q$.
	If $x\in U^\nu$, we observe $f_*T_xP^\nu+T_{f(x)}S^\nu\leq T_{f(x)}Q^\nu$,
	and $f_*(J(T_xP^\nu))+J(T_{f(x)}S^\nu)\leq J(T_{f(x)}Q^\nu)$,
	since $f$ is $(\nu,J)$--equivariant.
	By Lemma \ref{nu_J_manifold_property},
	we infer $f_*T_xP^\nu+T_{f(x)}S^\nu=T_{f(x)}Q^\nu$ and 
	$f_*(J(T_xP^\nu))+J(T_{f(x)}S^\nu)= J(T_{f(x)}Q^\nu)$.
	In particular, 
	$f^\nu\colon P^\nu\to Q^\nu$ is transverse to $S^\nu$
	on the open subset $U^\nu$ of $P^\nu$.	
	
	By the transversality assumption,
	$f^{-1}(S)\cap U$ is a closed smooth submanifold in $U$ of codimension $k$.
	Since $f$ is $(\nu,J)$--equivariant and $U$ is $\nu$--invariant,
	$f^{-1}(S)\cap U$ is also $\nu$--invariant, with fixed locus $f^{-1}(S)\cap U^\nu$,
	and the tangent bundle of $f^{-1}\cap U$ over $f^{-1}(S)\cap U^\nu$	is $J$--invariant.
	This shows that $f^{-1}(S)\cap U$ is a closed smooth $(\nu,J)$--submanifold in $U$	of codimension $k$.
\end{proof}

\begin{lemma}\label{nu_J_transverse_approximation}
	Let $f\colon P\to Q$ is a $(\nu,J)$--equivariant smooth map 
	between smooth $(\nu,J)$--manifolds $P$ and $Q$.
	Let $S$ is a closed smooth $(\nu,J)$--submsnifold in $Q$ of codimension $k$.
	
	Then, for any $\nu$--invariant compact subset $K\subset P$,
	there exist some $(\nu,J)$--equivariant smooth map $g\colon P\to Q$, 
	such that $g$ is transverse to $S$ on $K$,
	and such that $g$ is $(\nu,J)$--equivariantly smoothly homotopic to $f$
	supported on some compact neighborhood of $K$ in $P$.
	Hence, $g$ is transverse to $S$ on some $\nu$--invariant open neighborhood of $K$.
	
	%Moreover, 
	%given any $\nu$--invariant open neighborhood $W$ of $K$ in $P$, 
	%and	given any open neighborhood $\mathscr{N}$ 
	%of $f\colon W\to Q$ in the Whitney $C^\infty$--topology,	
	%one may require $g\in\mathscr{N}$, 
	%and $g=f$ outside some $\nu$--invariant compact neighborhood of $K$ in $W$.
	Moreover, 
	given any open neighborhood $\mathscr{N}$ of $f\colon P\to Q$ in the Whitney $C^\infty$--topology,	
	one may require $g\in\mathscr{N}$.
\end{lemma}

\begin{proof}
	Let $K$ be any $\nu$--invariant compact subset of $P$.
	We denote $K^\nu=K\cap P^\nu$.
	The essential part of the proof is to perturb $f$
	into a $(\nu,J)$--equivariant smooth map transverse to $S$ on $K^\nu$,
	which relies crucially on the structure $J$.
	We first give a careful construction for this step,
	and then finish the proof by an easy further purturbation.
	%Given any threshold neighborhoods $\mathscr{N}$ and $W$,
	%we replace $P$ with $W$ before performing the constructions below,
	%and make sure all the small perturbations below are within $\mathscr{N}$.
	All the small perturbations below can be done within 
	any given threshold neighborhood $\mathscr{N}$.
			
	Possibly after replacing $P$ by some paracompact open neighborhood of $K$,
	there exists some $(\nu,J)$--invariant smooth Riemannian metric on $P$
	(that is, $\nu_*$--invariant on $TP$ and $J$--invariant on $TP|_{P^\nu}$).
	Indeed, any smooth Riemannian metric on $P^\nu$	extends uniquely to be
	a $J$--invariant and $\nu_*$--invariant (Euclidean) inner product on $TP|_{P^\nu}$
	(Lemma \ref{nu_J_manifold_property}),
	and further extends to be some $\nu$--invariant smooth Riemannian metric on $P$.
	Then, automatically, $J(TP^\nu)\to P^\nu$ becomes the normal vector bundle of $P^\nu$ in $P$.
	The bundle space $J(TP^\nu)$ inherits a natural $(\nu,J)$--structure,
	(namely, the derivative $(\nu_*,J_*)$ of the $(\nu,J)$--structure of $P$,
	restricted to $J(TP^\nu)$, where $\nu_*$ fixes exactly the zero section).	
	Hence, the orthogonal Riemannian exponential map along $P^\nu$,
	denoted as $\mathrm{Exp}_\perp\colon J(TP^\nu)\to P$,	
	is a $(\nu,J)$--equivariant smooth map.
	Possibly after replacing $P$ by a smaller $\nu$--invariant open neighborhoods of $K$
	and rescaling the $(\nu,J)$--invariant Riemannian metric,
	we may assume $\mathrm{Exp}_\perp$
	to be an embedding restricted to the open unit disk bundle contained in $J(TP^\nu)$.
	
	Fix an auxiliary $(\nu,J)$--invariant smooth Riemannian metric as above for $P$, and similarly for $Q$.
	For any $\delta>0$, we denote by $\mathrm{Nhd}_\delta(K^\nu)$ 
	the distance--$\delta$ open neighborhood of $K^\nu$ in $P$.
	
	Denote by $\varphi\colon J(TP^\nu)\to J(TP^\nu)$
	the vector bundle automorphism $\varphi=-J\circ f^\nu_*\circ J$,
	where $f^\nu_*\colon TP^\nu\to TP^\nu$ denotes the tangent map of $f^\nu\colon P^\nu\to P^\nu$.
	Hence, $\varphi$ is a $(\nu,J)$--invariant self-diffeomorphism of the bundle space $J(TP^\nu)$.
	The $(\nu,J)$--equivariant smooth map 
	$$\tilde{f}=\mathrm{Exp}_\perp\circ\varphi\circ\mathrm{Exp}_\perp^{-1}$$
	is well-defined on an open neighborhood of $P^\nu$ in $P$,
	with image in $Q$,
	and $\tilde{f}$ is uniformly $C^0$--close to $f$ near $K^\nu$.
	In particular, for any sufficiently small $\epsilon>0$,
	$f$ and $\tilde{f}$ are $(\nu,J)$--equivariantly smoothly homotopic
	restricted to the open subset $\mathrm{Nhd}_\epsilon(K^\nu)$ of $P$,
	as $(\nu,J)$--equivariant smooth maps to $Q$.
	For example, a $(\nu,J)$--equivariant smooth homotopy can be constructed by 
	moving along the shortest geodesic segments connecting $f(x)$ and $\tilde{f}(x)$ in $Q$.	
	
	For any sufficiently small $\epsilon>0$,
	there exists some a smooth map $\tilde{g}^\nu_0\colon P^\nu\to Q^\nu$,
	such that $\tilde{g}^\nu_0$ is smoothly homotopic to $f^\nu$,
	and such that $\tilde{g}^\nu_0$ is transverse to $S^\nu$ on $\mathrm{Nhd}_{\epsilon}(K^\nu)\cap P^\nu$.
	The existence of $\tilde{g}^\nu_0$ follows 
	from standard differential topology \cite[Chapter 3, Theorem 2.1]{Hirsch_book}.
	This yields a $(\nu,J)$--equivariant smooth map 
	$$\tilde{g}_0=\mathrm{Exp}_\perp\circ\psi\circ\mathrm{Exp}_\perp^{-1},$$
	well-defined on some open neighborhood of $P^\nu$ in $P$,
	where $\psi\colon J(TP^\nu)\to J(TP^\nu)$ denotes $-J\circ \tilde{g}^\nu_{0*}\circ J$.
	It follows that $\tilde{g}_0$ is transverse to $S$ on $K^\nu$.
	Moreover, $\tilde{g}_0$ is $(\nu,J)$--equivariantly homotopic to $\tilde{f}$ 
	restricted to some open neighborhood of $P^\nu$ in $P$,
	as the expression for defining $\tilde{g}_0$ 
	applies similarly along any smooth homotopy of $f^\nu$ to $\tilde{g}^\nu_0$,
	
	Obtain some $(\nu,J)$--equivariant smooth map $\tilde{g}_0$,
	defined nearby $P^\nu\subset P$, as above.
	Take some sufficiently small $\epsilon>0$,
	such that $\tilde{g}_0$ is $(\nu,J)$--equivariantly, smoothly homotopic to $f$
	restricted to $\mathrm{Nhd}_\epsilon(K^\nu)\subset P$.
	Take some $(\nu,J)$--equivariant smooth homotopy
	$$H\colon \mathrm{Nhd}_\epsilon(K^\nu)\times [0,1]\to Q$$
	from $f$ to $\tilde{g}_0$ restricted to $\mathrm{Nhd}_\epsilon(K^\nu)$
	(namely, $H(x,0)=f(x)$ and $H(x,1)=\tilde{g}_0(x)$ for all $x\in\mathrm{Nhd}_\epsilon(K^\nu)$).
	Take a smooth function 
	$$\lambda\colon P\to [0,1]$$
	such that $\lambda$ is constant $0$ on $P\setminus\mathrm{Nhd}_{2\epsilon/3}(K^\nu)$,
	and constant $1$ on $\mathrm{Nhd}_{\epsilon/3}(K^\nu)$.
	We obtain a $(\nu,J)$--equivariant smooth map $g_0\colon P\to Q$
	by setting
	$$g_0(x)=H(x,\lambda(x)).$$
	
	By construction,
	the $(\nu,J)$--equivariant smooth map $g_0\colon P\to Q$
	is transverse to $S$ on $K^\nu$, and identical to $f$ 
	beyond distance $\epsilon$ from $K^\nu$ in $P$,
	and $(\nu,J)$--equivariantly smoothly homotopic to $f\colon P\to Q$.
	Moreover, as $\epsilon$ can be chosen arbitrarily small,
	$g_0$ can be arbitrarily close to $f$ with respect to the Whitney $C^\infty$--topology.
	
	For the special case $K=K^\nu$, 
	our proof is done,
	setting $g=g_0$ as above.		
	
	To address the general case, 
	we further perturb $g_0$ without affecting the obtained map nearby $K^\nu$.
	To this end, observe that $g_0$ is also transverse to $S$ on 
	the $\nu$--invariant open neighborhood $\mathrm{Nhd}_{\epsilon'}(K^\nu)$
	of $K^\nu$ in $P$, for some sufficiently small $\epsilon'>0$.
	We perturb $g_0$ as follows.
	
	Denote $K'=K\setminus \mathrm{Nhd}_{\epsilon'}(K^\nu)$.
	Hence, $K'$ is a $\nu$--invariant compact subset 
	in the open subset $P\setminus P^\nu$ in $P$ of freely discontinuous $\nu$--action.
	By compactness,
	we can construct finitely many open subsets 
	$U_1,\ldots,U_m$ in $P\setminus P^\nu$,
	such that $K'$ is contained in the union of all $U_i\cup \nu(U_i)$,
	and such that each $U_i\cap \nu(U_i)$ is empty.
	For each $i=1,\ldots,m$, 
	denote $K'_i=K'\setminus\bigcup_{j\neq i}(U_j\cup \nu(U_j))$.
	Hence, each $K'_i$ is a $\nu$--invariant compact subset in $U_i\cup\nu(U_i)$.
	By standard differential topology \cite[Chapter 3, Theorem 2.1]{Hirsch_book}, 
	we can perturb $g_0$ by smooth homotopy, 
	supported within $U_1$, obtaining a smooth map transverse to $S$ 
	on some open neighborhood of $K'_1$ (in $U_1$) and on $\mathrm{Nhd}_\epsilon(K^\nu)\cap U_1$.
	Since $U_1$ is disjoint from $\nu(U_1)$, 
	the modification supported within $U_1$ 
	determines a unique $(\nu,J)$--equivariant smooth perturbation
	$g_1\colon P\to Q$ of $f$, 
	by $(\nu,J)$--equivariant smooth homotopy,
	such that $g_1$ is transverse to $S$ 
	on some $\nu$--invariant open neighborhood of $K'_1$ (in $U_1\cup\nu(U_1)$) and on $\mathrm{Nhd}_{\epsilon'}(K^\nu)$.
	Similarly, 
	a sufficiently small $(\nu,J)$--equivariant perturbation of $g_1$ supported within $U_2$ 
	yields some $g_2\colon P\to Q$ transverse to $S$ on
	some $\nu$--invariant open neighborhood of $K'_1\cup K'_2$ and on $\mathrm{Nhd}_{\epsilon'}(K^\nu)$.
	Proceed similarly to obtain $(\nu,J)$--equivariant perturbations $g_3,g_4,\ldots$ one by one,
	such that each $g_i\colon P\to Q$ is transverse to $S$ 
	on some $\nu$--invariant neighborhood of $K'_1\cup\cdots\cup K'_i$
	and on $\mathrm{Nhd}_\epsilon(K^\nu)$.
	Note that each $g_i$ keeps identical to $g_0$ on some $\nu$--invariant neighborhood of $K^\nu$ in $P$.
	In the end, set $g=g_m$. It is clear that $g\colon P\to Q$ satisfies all the desired properties.
\end{proof}

\begin{remark}\label{equivariant_transversality_literature}
	In the literature, there are many works on equivariant differential topology
	for smooth $G$--manifolds, where $G$ is assumed to be a compact Lie group.
	We refer the reader to 
	Wasserman \cite{Wasserman_edt} for an exposition of the general theory.
	The notion of $G$--equivariant transversality is somewhat subtle,
	as studied in depth by Bierstone \cite{Bierstone_equivariant} and by Field \cite{Field_transversality}.
	Even in a most naive sense, the transverse intersection of $G$--submanifolds 
	is not necessarily a $G$--submanifold, but only a stratified $G$--subspace,
	simply due to the dimension reason. 
	
	Our $(\nu,J)$--manifolds are not $G$--manifolds, because of the extra data $J$.
	In the $(\nu,J)$--equivariant setting,
	transverse preimage and transverse approximation are more well behaved,
	again because of $J$, 
	as is clear from the proofs of Lemmas \ref{nu_J_transverse_property} and \ref{nu_J_transverse_approximation}.
\end{remark}

\section{Deforming an abelian elliptic representation}\label{Sec-deform}

In this section, we prove Theorem \ref{main_deform}.
We first restate it more precisely, as Theorem \ref{deform_character} 
and Corollary \ref{deform_representation} below.
The derivation of Corollary \ref{deform_representation} from Theorem \ref{deform_character}
is well-known, and we include a proof for the reader's reference.

\begin{theorem}\label{deform_character}
	Let $K\subset S^3$ be a knot.
	Denote by $\Delta_K(t)$ the Alexander polynomial of $K$.
	Suppose $\theta_0\in(0,\pi)$.
	Denote by $\chi_0\colon\pi_1(S^3\setminus K)\to\Real$ 
	the unique reducible $\mathrm{SL}(2,\Real)$--character
	with value $2\cos\theta_0$ at the meridian of $K$.
		
	If $e^{2\cdot\theta_0\sqrt{-1}}$ is a complex zero of odd order for $\Delta_K(t)$,
	%the Alexander polynomial $\Delta_K(t)$ of $K$,
	then, $\chi_0$
	extends to some continuous path of $\mathrm{SL}(2,\Real)$--characters
	$\chi_s\colon \pi_1(S^3\setminus K)\to \Real$, 
	parametrized by $s\in[0,1]$,
	such that $\chi_s$ is irreducible for all $s\in(0,1]$.
\end{theorem}

%The conclusion of Theorem \ref{deform_character} can be promoted into the following form.

\begin{corollary}\label{deform_representation}
	With notations and assumptions in Theorem \ref{deform_character},
	there exists some continuous path of representations
	$\rho_s\colon \pi_1(S^3\setminus K)\to \mathrm{SL}(2,\Real)$,
	parametrized by $s\in[0,1]$, such that $\rho_0$ has character $\chi_0$,
	and such that $\rho_s$ is irreducible for all $s\in(0,1]$.
	In particular, the path $\rho_s$ lifts to a continuous path of representations
	$\tilde{\rho}_s\colon \pi_1(S^3\setminus K)\to \widetilde{\mathrm{SL}}(2,\Real)$.
\end{corollary}

\begin{proof}
Obtain a path of $\mathrm{SL}(2,\Real)$--characters 
$\chi_s\colon\pi_1(S^3\setminus K)\to \Real$ as in Theorem \ref{deform_character}.
By the weak path lifting property of the continuous map
$\mathcal{R}(\pi_1(S^3\setminus K),\mathrm{SL}(2,\Real))\to \mathcal{X}(\pi_1(S^3\setminus K),\mathrm{SL}(2,\Real))$
(see Subsection \ref{Subsec_analytic}),
we can lift the path $\chi_s$ to obtain
some path $\rho_s\colon\pi_1(S^3\setminus K)\to \mathrm{SL}(2,\Real)$,
which is as asserted.

Note that the $\rho_0$--image of the meridian is noncentral and elliptic
(with eigenvalues $e^{\pm\theta_0\sqrt{-1}}$),
so $\rho_0$ factors through the abelianization of $\pi_1(S^3\setminus K)$,
which is infinite cyclic.
Therefore, $\rho_0$ lifts to a representation 
$\tilde{\rho}_0\colon \pi_1(S^3\setminus K)\to \widetilde{\mathrm{SL}}(2,\Real)$.
The unique path lifting property of the covering projection
$\widetilde{\mathrm{SL}}(2,\Real)\to \mathrm{SL}(2,\Real)$
implies that $\tilde{\rho}_0$ extends to be a unique path $\tilde{\rho}_s$,
lifting the path $\rho_s$ with the lifted base point $\tilde{\rho}_0$.
To be more precise, for any element $g\in\pi_1(S^3\setminus K)$,
the path $\tilde{\rho}_s(g)$ in $\widetilde{\mathrm{SL}}(2,\Real)$ is
the unique lift of the path $\rho_s(g)$ in $\mathrm{SL}(2,\Real)$
with the lifted base point $\tilde{\rho}_0(g)$.
The unique path lifting property also makes sure that any relation between the group elements is preserved.
Hence, the path $\rho_s$ can be lifted to a path $\tilde{\rho}_s$ as asserted,
and indeed,
the lift depends only on a chosen lift of the $\rho_0$--image of the meridian.
\end{proof}

The rest of this section is devoted to the proof of Theorem \ref{deform_character}.

We first recall the notion of local intersection number in standard differential topology.
Suppose that $P$ and $Q$ are manifolds (without boundary and of designated dimension).
Suppose that $S$ is an closed smooth submanifold in $Q$ of codimension equal to the dimension of $P$.
Suppose that $U$ is an open subset in $P$.
For any smooth map $f\colon P\to Q$, such that $f^{-1}(S)\cap U$ is compact,
the \emph{local modulo $2$ intersection number} of $f$ with $S$ on $U$ 
is a well-defined integer modulo $2$,
denoted as
$$I_2(f,S;U)\in\Integral/2\Integral.$$
To be precise, $I_2(f,S;U)$ is defined to be the parity count of $g^{-1}(S)\cap U$,
where $g\colon P\to Q$ is any smooth map smoothly homotopic to $f$ 
supported on a compact neighborhood of $f^{-1}(S)\cap U$ in $U$, 
such that $g$ is a smooth map transverse to $S$ on $U$.
The compactness and the transversality imply that $g^{-1}(S)\cap U$
is a finite set of points.
The resulting parity of the number of points in $g^{-1}(S)\cap U$ 
does not depend on the choice of $g$,
subject to the required conditions.
See \cite[Chapter 5, Section 2]{Hirsch_book} for a reference.

%
%The sign is assigned to be positive at a point $x\in g^{-1}(S)\cap U$
%if the orientation of $g_*T_xP\oplus T_{g(x)}S$ agrees with the orientation of $T_{g(x)}Q$,
%or negative otherwise.

\begin{lemma}\label{local_data}
	Let $\sigma\in\mathscr{B}_n$ be any braid group element.
	Suppose that $T\subset\Complex\setminus\{\pm2\}$
	is some complex-conjugation invariant, open subset.
	Adopt Notations \ref{notation_braid} and \ref{notation_slice}.
	Denote 
	$\mathcal{P}_T=\mathcal{X}^{\mathtt{irr}}_T\left(F_n,\mathrm{SL}(2,\Complex)\right)$,
	and
	$\mathcal{Q}_T=\mathcal{X}^{\mathtt{irr}}_T\left({WF}_n,\mathrm{SL}(2,\Complex)\right)$.
	For any $\tau\in T$, denote by	$\mathcal{S}_\tau$ the image of 
	$\mathcal{X}^{\mathtt{irr}}_\tau\left(F_n,\mathrm{SL}(2,\Complex)\right)$
	under $r^*_{\mathrm{id}}\colon\mathcal{P}_T\to \mathcal{Q}_T$.
	View $T$, $\mathcal{P}_T$, and $\mathcal{Q}_T$ as 
	$(\nu,J)$--smooth manifolds,
	%$(\nu,J)$--equivariant smooth bundles over $T$,
	and $\mathcal{S}_\tau$ as a closed smooth submanifold of $\mathcal{Q}_T$,
	according to Example \ref{nu_J_example}.
	
	For any $\tau_0\in T$,
	there exists some tuple $(f_T,\mathcal{U}_T,D)$,
	depending on $(\sigma,T,\tau_0)$, as follows.
	\begin{itemize}
	\item
	The item $f_T\colon \mathcal{P}_T\to\mathcal{Q}_T$ is a $(\nu,J)$--equivariant smooth map,
	such that $f_T$ is $(\nu,J)$--equivariantly smoothly homotopic to 
	$r^*_\sigma\colon\mathcal{P}_T\to\mathcal{Q}_T$ supported 
	on some $\nu$--invaraint compact subset in $\mathcal{P}_T$.
	\item
	The item $\mathcal{U}_T\subset\mathcal{P}_T$ is 
	a $\nu$--invariant open subset in $\mathcal{P}_T$,
	such that $\mathcal{U}_T$ contains $\mathcal{X}^{\mathtt{irr}}_T(F_n,\mathrm{SU}(2))$,
	and such that $\mathcal{U}_T\cup \mathcal{X}^{\mathtt{red}}_T(F_n,\mathrm{SL}(2,\Complex))$
	is open in $\mathcal{X}_T(F_n,\mathrm{SL}(2,\Complex))$.	
	\item
	The item $D\subset T$ is a $\nu$--invariant open neighborhood of $\tau_0$,
	such that, for any $\tau\in D$, 
	the closure of $f^{-1}_T(\mathcal{S}_\tau)\cap\mathcal{U}_T$ 
	in $\mathcal{U}_T\cup\mathcal{X}^{\mathtt{red}}_T(F_n,\mathrm{SL}(2,\Complex))$
	is compact.
	\end{itemize}
	
	Moreover,
	given any complex-conjugation invariant, open neighborhood $\mathcal{W}_T$ of
	$\mathcal{X}^{\mathtt{red}}_T(F_n,\mathrm{SL}(2,\Complex))\cup \mathcal{X}_T(F_n,\mathrm{SU}(2))$
	in $\mathcal{X}_T(F_n,\mathrm{SL}(2,\Complex))$,	
	and given any $\nu$--invariant, closed subset $\mathcal{V}_T$ in $\mathcal{Q}_T$
	with $r^*_\sigma(\mathcal{P}_T\cap\mathcal{W}_T)\cap\mathcal{V}_T=\emptyset$,
	one may require 
	$f_T=r_\sigma^*$ outside some compact subset in $\mathcal{P}_T\cap\mathcal{W}_T$,
	and $f_T(\mathcal{P}_T\cap\mathcal{W}_T)\cap\mathcal{V}_T=\emptyset$,
	and $\mathcal{U}_T\subset \mathcal{P}_T\cap\mathcal{W}_T$.
\end{lemma}

\begin{proof}
	Given $\mathcal{W}_T$ and $\mathcal{V}_T$ as assumed,
	we construct as follows.
	
	Fix some auxiliary $\nu$--invariant compact neighborhood of $\tau_0$ in $T$, denoted as
	$$T_0\subset T.$$
	Since the parametric trace $T_0$--slice
	$\mathcal{X}^{\mathtt{red}}_{T_0}(F_n,\mathrm{SL}(2,\Complex))$ is compact (Lemma \ref{parametric_slices}),
	%(Example \ref{Subsec_analytic}, Notation \ref{notation_slice}, and Lemma \ref{parametric_slices}).
	we can construct some auxiliary proper continuous function 
	$$\lambda\colon \mathcal{X}_{T_0}(F_n,\mathrm{SL}(2,\Complex))\to[0,+\infty),$$
	such that 
	$\lambda$ is invariant under complex conjugation of $\mathrm{SL}(2,\Complex)$--characters,
	and such that
	$\lambda$ is equal to $0$ on $\mathcal{X}^{\mathtt{red}}_{T_0}(F_n,\mathrm{SL}(2,\Complex))$.
	For any subset $J\subset[0,+\infty)$, we denote 
	$$\mathcal{E}_{J}=\lambda^{-1}(J).$$
	Note that for any closed subset $J\subset(0,+\infty)$,
	$\mathcal{E}_{J}$ is a $\nu$--invariant closed subset in $\mathcal{P}_T$,
	compact if $J$ is compact.
	Since the parametric trace $T_0$--slice 
	$\mathcal{X}_{T_0}(F_n,\mathrm{SU}(2))$ is compact
	(see Subsection \ref{Subsec_analytic} and Notation \ref{notation_slice}),
	we assume without loss of generality that $\mathcal{E}_{[0,1]}$
	contains $\mathcal{X}_{T_0}(F_n,\mathrm{SU}(2))$,
	and $\mathcal{E}_{[0,3]}$ is contained in $\mathcal{W}_T$.
	
	By $(\nu,J)$--equivariant perturbation (Lemma \ref{nu_J_transverse_approximation}), 
	we obtain some $(\nu,J)$--equivariant proper smooth map,
	denoted as
	$$f_T\colon\mathcal{P}_T\to \mathcal{Q}_T,$$
	such that $f_T$ is transverse to $\mathcal{S}_{\tau_0}$ 
	on the compact subset $\mathcal{E}_{[2,3]}\subset\mathcal{P}_T\cap\mathcal{W}_T$.
	Moreover, we can make sure that
	$f_T$ is $(\nu,J)$--equivariantly smoothly homotopic to $r^*_\sigma$,
	supported on 
	some $\nu$--invariant compact neighborhood of $\mathcal{E}_{[2,3]}$ in $\mathcal{P}_T\cap\mathcal{W}_T$,
	such that the image of the support under $f_T$ 
	stays disjoint from $\mathcal{V}_T$ in $\mathcal{Q}_T$. 
	
	Note that $\mathcal{S}_{\tau_0}$ has codimension $(8n-10)-(4n-6)=4n-4$ in $\mathcal{Q}_T$,
	which is equal to the dimension of $\mathcal{P}_T$ (Lemmas \ref{slice_regularity_F} and \ref{parametric_slices}).
	Therefore, $f^{-1}_T(\mathcal{S}_{\tau_0})$ intersects some $\nu$--invariant open neighborhood of $\mathcal{E}_{[2,3]}$
	as a closed $(\nu,J)$--submanifold of dimension $0$ therein.
	The compactness of $\mathcal{E}_{[2,3]}$ implies that
	$f^{-1}_T(\mathcal{S}_{\tau_0})\cap\mathcal{E}_{[2,3]}$ is a finite set of points.
	Obtain some short compact interval, denoted as
	$$[a_0,b_0]\in(2,3),$$
	such that $f^{-1}_T(\mathcal{S}_{\tau_0})\cap\mathcal{E}_{[a_0,b_0]}$ is empty.
	By transversality and $(\nu,J)$--equivariance,
	there exists for some sufficiently small $\nu$--invariant open neighborhood of $\tau_0$, 
	denoted as $$D\subset T,$$
	such that for any $\tau\in D$,
	$f_T\colon\mathcal{P}_T\to\mathcal{Q}_T$ 
	is transverse to $\mathcal{S}_\tau$ on $\mathcal{E}_{[2,3]}$,
	and such that 
	$f^{-1}_T(\mathcal{S}_{\tau})\cap\mathcal{E}_{[a_0,b_0]}$ is empty.
	We obtain a $\nu$--invariant open subset of $\mathcal{P}_T$, denoted as
	$$\mathcal{U}_T\subset\mathcal{P}_T,$$
	by setting 
	$\mathcal{U}_T=(\mathcal{P}_T\cap\mathcal{W}_T)\setminus\mathcal{E}_{[b_0,+\infty)}$.
	
	We verify the asserted properties for the tuple $(f_T,\mathcal{U}_T,D)$, as follows.	
	The asserted properties regarding $f_T$ is obvious by construction.
	The union of $\mathcal{U}_T$ and $\mathcal{X}^{\mathtt{red}}_T(F_n,\mathrm{SL}(2,\Complex))$
	is an open neighborhood of 
	$\mathcal{X}^{\mathtt{red}}_T(F_n,\mathrm{SL}(2,\Complex))\cup \mathcal{X}_T(F_n,\mathrm{SU}(2))$
	in $\mathcal{X}_T(F_n,\mathrm{SL}(2,\Complex))$.
	Indeed, its complement is the union of the closed subset $\mathcal{P}_T\setminus \mathcal{W}_T$
	with $\mathcal{E}_{[b_0,+\infty)}$,
	where $\mathcal{E}_{[b_0,+\infty)}$ is 
	closed in the closed subset $\mathcal{X}_{V_0}(F_n,\mathrm{SL}(2,\Complex))$,
	and $\mathcal{E}_{[b_0,+\infty)}$ has empty intersection with 
	$\mathcal{X}_T(F_n,\mathrm{SU}(2))$.
	Observe that $f_T\colon\mathcal{P}_T\to\mathcal{Q}_T$ extends to be a continuous map 
	$\mathcal{X}_T(F_n,\mathrm{SL}(2,\Complex))\to \mathcal{X}_T(F_n,\mathrm{SL}(2,\Complex))$,
	agreeing with $r^*_\sigma\colon \mathcal{X}_T(F_n,\mathrm{SL}(2,\Complex))\to \mathcal{X}_T(F_n,\mathrm{SL}(2,\Complex))$ 
	on the compact subset $\mathcal{E}_{[0,1]}$.
	For any $\tau\in D$,
	the closure of $f^{-1}_T(\mathcal{S}_\tau)\cap\mathcal{U}_T$ in $\mathcal{X}_T(F_n,\mathrm{SL}(2,\Complex))$
	is contained in the uniform compact subset $\mathcal{E}_{[0,a_0]}$,
	which is contained in $\mathcal{U}_T\cup \mathcal{X}^{\mathtt{red}}_T(F_n,\mathrm{SL}(2,\Complex))$.
	Hence, $f^{-1}_T(\mathcal{S}_\tau)\cap\mathcal{U}_T$ has compact closure
	$\mathcal{U}_T\cup \mathcal{X}^{\mathtt{red}}_T(F_n,\mathrm{SL}(2,\Complex))$.
	Therefore, the tuple $(f_T,\mathcal{U}_T,D)$ is as desired.
	Moreover, the additional requirements 
	with respect to $\mathcal{W}_T$ and $\mathcal{V}_T$ are also satisfied.
\end{proof}

\begin{lemma}\label{not_zero}
	Let $K\subset S^3$ be a knot.
	Denote by $\Delta_K(t)$ the Alexander polynomial of $K$.
	Fix some orientation of $K$ and some braid group element $\sigma\in\mathscr{B}_n$
	for presenting $\pi_1(S^3\setminus K)$, as declared in Lemma \ref{make_braid_presentation}.
	Let $(f_T,\mathcal{U}_T,D)$ be any tuple with respect to some $(T,\tau_0)$, 
	as declared in Lemma \ref{local_data}.
	Retain the notations $\mathcal{P}_T$, $\mathcal{Q}_T$, and $\mathcal{S}_T$ in Lemma \ref{local_data}.
	
	For any $\tau\in D$, 
	and for any $\lambda\in\Complex\setminus\{0\}$ satisfying $\lambda+\lambda^{-1}=\tau$,
	the following statement holds.
	If $\Delta_K\left(\lambda^2\right)\neq0$,
	then $f^{-1}_T(\mathcal{S}_\tau)\cap\mathcal{U}_T$ is compact.
\end{lemma}

\begin{proof}
	For any $\tau\in D$,
	the properties in Lemma \ref{local_data} guarantees that
	the subset $f^{-1}_T(\mathcal{S}_\tau)\cap\mathcal{U}_T$ is closed in $\mathcal{U}_T$,
	and is precompact in
	$\mathcal{U}_T\cup \mathcal{X}^{\mathtt{red}}_T(F_n,\mathrm{SL}(2,\Complex))$.
	It will follow that $f^{-1}_T(\mathcal{S}_\tau)\cap\mathcal{U}_T$ is compact,
	if $f^{-1}_T(\mathcal{S}_\tau)\cap\mathcal{U}_T$ is isolated from 
	the closed subset $\mathcal{X}^{\mathtt{red}}_T(F_n,\mathrm{SL}(2,\Complex))$
	in $\mathcal{U}_T\cup \mathcal{X}^{\mathtt{red}}_T(F_n,\mathrm{SL}(2,\Complex))$.
	This is where the assumptions 
	$\tau=\lambda+\lambda^{-1}$ and $\Delta_K(\lambda^2)\neq0$ come into play.
	
	To this end, we observe the following identifications,
	according to Lemma \ref{make_braid_presentation} 
	and Notations \ref{notation_braid} and \ref{notation_slice}.
	By pulling back along the group homomorphism $p_\sigma\colon WF_\sigma\to \Pi_\sigma$,
	we can identify $\mathcal{X}(\Pi_\sigma;\mathrm{SL}(2,\Complex))$ 
	with the intersection in $\mathcal{X}({WF}_n,\mathrm{SL}(2,\Complex))$ of the images of
	$\mathcal{X}(F_n,\mathrm{SL}(2,\Complex))$
	under the proper embeddings $r^*_\sigma$ and $r^*_{\mathrm{id}}$.
	In particular, for any $\tau\in T$,
	the intersection in $\mathcal{Q}_T$ of $r^*_\sigma(\mathcal{P}_T)$ with 
	$\mathcal{S}_\tau$ can be identified as 
	the subset of $\mathrm{SL}(2,\Complex)$--characters $\chi\colon\Pi_\sigma\to \Complex$
	determined by $\chi(x_1)=\cdots=\chi(x_n)=\tau$	
	(see Lemma \ref{local_data} and Notation \ref{notation_braid}).
	This is equivalent to saying that $\chi$ is 
	the value of $\chi\colon\pi_1(S^3\setminus K)\to \Complex$ 
	at the meridian of $K$ is equal to $\tau$,
	according to the identification $\pi_1(S^3\setminus K)\cong\Pi_\sigma$
	(see Lemma \ref{make_braid_presentation}).
	For any $\tau\in T$, 
	there is a unique reducible $\mathrm{SL}(2,\Complex)$--character of $\Pi_\sigma$
	with the above property, which we denote as 
	$$\chi_\tau\in\mathcal{X}^{\mathtt{red}}(\Pi_\sigma,\mathrm{SL}(2,\Complex)).$$
	
	Suppose $\tau\in T$.
	If $\tau=\lambda+\lambda^{-1}$ for some complex $\lambda\neq0$,
	and if $\Delta_K(\lambda^2)\neq0$,
	then some open neighborhood of $\chi_\tau$ in $\mathcal{X}(\Pi_\sigma,\mathrm{SL}(2,\Complex))$
	contains no irreducible $\mathrm{SL}(\Complex)$--characters.
	In fact, the condition $\Delta_K(\lambda^2)\neq0$ implies that
	any representation 
	$\rho_\tau\colon\Pi_\sigma\to\mathrm{SL}(2,\Complex)$ with character $\chi_\tau$
	is abelian (and hence, unique up to conjugation),
	by a theorem due to Burde \cite{Burde_deform} and due to de Rham \cite{de_Rham_deform}
	(see also \cite[Corollary 4.3]{HPS_reducible}).
	(Besides, the theorem shows 
	the existence of 
	some reducible, non-abelian $\mathrm{SL}(2,\Complex)$--representation with character $\chi_\tau$,
	if $\Delta_K(\lambda^2)=0$.)
	The assumption $T\subset\Complex\setminus\{\pm2\}$ implies that 
	$\rho_\tau$ is conjugate to a diagonal representation.
	In this case, 
	$\rho_\tau$ is a smooth point (of local complex dimension $3$) 
	in $\mathcal{R}(\Pi_\sigma,\mathrm{SL}(2,\Complex)$,
	and sufficiently nearby representations are all abelian, 
	by a theorem due to Heusener, Porti, and Suar\'es-Peir\'o \cite[Theorem 2.7]{HPS_reducible}.
	This implies that $\chi_\tau$ is 
	a smooth point (of local complex dimension $1$) in $\mathcal{X}(\Pi_\sigma,\mathrm{SL}(2,\Complex))$,
	and sufficiently nearby characters are all reducible.
	
	In particular, 
	under the assumptions $\tau=\lambda+\lambda^{-1}$ and $\Delta_K(\lambda^2)\neq0$,
	the intersection 
	of $r^*_\sigma(\mathcal{X}^{\mathtt{red}}_T(F_n,\mathrm{SL}(2,\Complex))$ 
	and $r^*_{\mathrm{id}}(\mathcal{X}^{\mathtt{red}}_\tau(F_n,\mathrm{SL}(2,\Complex))$ 
	is empty in $\mathcal{X}_T(F_n,\mathrm{SL}(2,\Complex)$
	(see Notation \ref{notation_slice}).
	Equivalently, the subset $((r^*_\sigma)^{-1}(\mathcal{S}_\tau))\cap\mathcal{P}_T$
	in $\mathcal{X}_T(F_n,\mathrm{SL}(2,\Complex))$ 
	is the same as 
	$(r^*_\sigma)^{-1}(r^*_{\mathrm{id}}(\mathcal{X}_\tau(F_n,\mathrm{SL}(2,\Complex))))$,
	which is closed in $\mathcal{X}_T(F_n,\mathrm{SL}(2,\Complex))$.
	Since $f_T\colon\mathcal{P}_T\to\mathcal{Q}_T$ 
	is identical to $r^*_\sigma$ outside some compact subset in $\mathcal{P}_T$
	(see Lemma \ref{local_data}),
	we infer that $f_T^{-1}(\mathcal{S}_\tau)\cap\mathcal{P}_T$ 
	is also closed in $\mathcal{X}_T(F_n,\mathrm{SL}(2,\Complex))$ (see Lemma \ref{local_data}).
	Intersecting with $\mathcal{U}_T\cup \mathcal{X}^{\mathtt{red}}_T(F_n,\mathrm{SL}(2,\Complex))$,
	we see that $f_T^{-1}(\mathcal{S}_\tau)\cap\mathcal{U}_T$
	is closed in $\mathcal{U}_T\cup \mathcal{X}^{\mathtt{red}}_T(F_n,\mathrm{SL}(2,\Complex))$.
	Hence, $f_T^{-1}(\mathcal{S}_\tau)\cap\mathcal{U}_T$ is isolated from the closed subset
	$\mathcal{X}^{\mathtt{red}}_T(F_n,\mathrm{SL}(2,\Complex))$.
	This implies the compactness of $f_T^{-1}(\mathcal{S}_\tau)\cap\mathcal{U}_T$,
	for any $\tau\in D$, as explained at the beginning of the proof.	
\end{proof}

\begin{lemma}\label{half_signature_parity}
	Let $K\subset S^3$ be a knot.
	Denote by $\Delta_K(t)$ the Alexander polynomial of $K$.
	Fix some orientation of $K$ and some braid group element $\sigma\in\mathscr{B}_n$
	for presenting $\pi_1(S^3\setminus K)$, as declared in Lemma \ref{make_braid_presentation}.
	Let $(f_T,\mathcal{U}_T,D)$ be any tuple with respect to some $(T,\tau_0)$, 
	as declared in Lemma \ref{local_data}.
	Retain the notations $\mathcal{P}_T$, $\mathcal{Q}_T$, and $\mathcal{S}_T$ in Lemma \ref{local_data}
	
	For any real $\tau\in D$ satisfying $|\tau|<2$, and 
	for any $\theta\in(0,\pi)$ satisfying $2\cos\theta=\tau$, the following statement holds.
	If $\Delta_K\left(e^{2\cdot\theta\cdot\sqrt{-1}}\right)\neq0$, 
	and if $f^{-1}_T(\mathcal{S}_{\tau})\cap\mathcal{U}_T$ 
	has empty intersection with $\mathcal{X}^{\mathtt{irr}}_T(F_n,\mathrm{SL}(2,\Real))$ in $\mathcal{P}_T$,
	then $$I_2(f_T,\mathcal{S}_\tau;\mathcal{U}_T)\equiv
	\frac{1}{2}\cdot\mathrm{sgn}_K\left(e^{2\cdot\theta\cdot\sqrt{-1}}\right)\bmod 2.$$
\end{lemma}

\begin{proof}
	The condition $\Delta_K\left(e^{2\cdot\theta\cdot\sqrt{-1}}\right)\neq0$ 
	implies that $f^{-1}_T(\mathcal{S}_{\tau})\cap\mathcal{U}_T$ is compact (Lemma \ref{not_zero}).
	In particular, 
	the local modulo $2$ intersection number $I_2(f_T,\mathcal{S}_\tau;\mathcal{U}_T)$ is well-defined.
	
	To compute $I_2(f_T,\mathcal{S}_\tau;\mathcal{U}_T)$,
	take some $(\nu,J)$--equivariantly smooth homotopic small perturbation of $f_T$,	
	supported in a compact neighborhood of $f^{-1}_T(\mathcal{S}_{\tau})\cap\mathcal{U}_T$ 
	within $\mathcal{U}_T$. 
	Denote the resulting $(\nu,J)$--equivariant map as $g_T\colon\mathcal{P}_T\to\mathcal{Q}_T$,
	then we can make sure that $g_T$ is transverse to $\mathcal{S}_\tau$ on $\mathcal{U}_T$ 
	(Lemma \ref{nu_J_transverse_approximation}).
	This also makes sure that $g^\nu_T\colon\mathcal{P}^\nu_T\to\mathcal{Q}^\nu_T$
	is transverse to $\mathcal{S}^\nu_\tau$ on $\mathcal{U}^\nu_T$ (Lemma \ref{nu_J_transverse_property}).
	Note that the fixed locus $\mathcal{Q}^\nu_T$ is the disjoint union of closed smooth submanifolds
	$\mathcal{X}^{\mathtt{irr}}(WF_n,\mathrm{SU}(2))$ and $\mathcal{X}^{\mathtt{irr}}(WF_n,\mathrm{SL}(2,\Real))$
	(Subsection \ref{Subsec_analytic} and Example \ref{nu_J_example}),
	and the $(\nu,J)$--equivariant smooth map $r^*_\sigma\colon\mathcal{P}_T\to\mathcal{Q}_T$
	must preserve the fixed loci, 
	mapping irreducible $\mathrm{SU}(2)$--characters to irreducible $\mathrm{SU}(2)$--characters,
	and irreducible $\mathrm{SL}(2,\Real)$--characters to irreducible $\mathrm{SL}(2,\Real)$--characters.
	The same property holds for $f_T$ and for $g_T$,
	as they are both $(\nu,J)$--equivariantly smoothly homotopic to $r^*_\sigma$.
	It follows that $g^{-1}_T(\mathcal{S}_\tau)\cap\mathcal{U}_T$
	is a $0$--dimensional, compact, smooth $(\nu,J)$--manifold,
	which consists of 
	finitely many irreducible $\mathrm{SU}(2)$--characters, and $\mathrm{SL}(2,\Real)$--characters,
	and other non-real, irreducible $\mathrm{SL}(2,\Complex)$--characters.
	The non-reals are grouped in pairs under the free $\nu$--involution.
	This yields the following count of the local modulo $2$ intersection number in $\Integral/2\Integral$.
	\begin{eqnarray*}
	I_2(f_T,\mathcal{S}_\tau;\mathcal{U}_\tau)&=&I_2(g_T,\mathcal{S}_\tau;\mathcal{U}_T)\\
	&=&I_2\left(g^\nu_T,\mathcal{S}^\nu_\tau;\mathcal{U}^\nu_T\right)\\
	&=&
	I_2\left(g^\nu_T,\mathcal{S}^\nu_\tau;\mathcal{U}^\nu_T\right)_{\mathrm{SU}(2)}+
	I_2\left(g^\nu_T,\mathcal{S}^\nu_\tau;\mathcal{U}^\nu_T\right)_{\mathrm{SL}(2,\Real)}
	\end{eqnarray*}
	Here, the terms in the last line refer to the separate contributions
	from the $\mathrm{SU}(2)$ type and from the $\mathrm{SL}(2,\Real)$ type,
	accordingly.
	
	The condition that $f^{-1}_T(\mathcal{S}_{\tau})\cap\mathcal{U}_T$ 
	has empty intersection with $\mathcal{X}^{\mathtt{irr}}_T(F_n,\mathrm{SL}(2,\Real))$ in $\mathcal{P}_T$
	implies the similar property for $g^{-1}_T(\mathcal{S}_{\tau})\cap\mathcal{U}_T$,
	as long as we make sure the pertubation to be sufficiently small.
	To be precise, this can be guaranteed if we require the image of the compact support of the perturbation
	to stay in the open subset in $\mathcal{Q}_T$
	complementary to $\mathcal{X}^{\mathtt{irr}}(WF_n,\mathrm{SL}(2,\Real))$.
	In this case,
	it follows that there are no contribution from the $\mathrm{SL}(2,\Real)$ type,
	namely,
	$$I_2\left(g^\nu_T,\mathcal{S}^\nu_\tau;\mathcal{U}^\nu_T\right)_{\mathrm{SL}(2,\Real)}=0.$$
	So, it remains to prove the following identity. 
	$$I_2\left(g^\nu_T,\mathcal{S}^\nu_\tau;\mathcal{U}^\nu_T\right)_{\mathrm{SU}(2)}
	\equiv\frac{1}{2}\cdot\mathrm{sgn}_K\left(e^{2\cdot\theta\cdot\sqrt{-1}}\right)\bmod 2.$$
	
	In fact, the last identity is direct application of 
	a theorem due to Heusener and Kroll \cite[Theorem 1.2]{Heusener--Kroll_abelian},
	which characterizes the value of $(1/2)\cdot\mathrm{sgn}_K$,
	at any $e^{2\cdot\theta\cdot\sqrt{-1}}$ other that zeros of $\Delta_K$, 
	as the intersection number 
	between trace $\tau$--slices of irreducible $\mathrm{SU}(2)$--characters.
	Their theorem generalizes a former characterization of the knot signature	in terms of trace $0$--slices,
	due to Lin \cite{Lin_invariant}.
	Below, we recall their theorem, and explain the adaptation.
	For accuracy, we expose with intersection numbers, as they do,
	although the modulo $2$ residues are what we only need.
		
	For simplicity, we rewrite $T'=T^\nu=T\cap\Real$,
	and	$\mathcal{P}'_{T'}=\mathcal{X}^{\mathtt{irr}}_{T'}(F_n,\mathrm{SU}(2))$,
	and $\mathcal{Q}'_{T'}=\mathcal{X}^{\mathtt{irr}}_{T'}(WF_n,\mathrm{SU}(2))$.
	Rewrite $g'_{T'}\colon\mathcal{P'}_{T'}\to\mathcal{Q}'_{T'}$ for the restriction of $g^\nu_T$ to $\mathcal{P}'_{T'}$,
	and $\mathcal{S}'_\tau$ for the intersection of $\mathcal{S}_\tau$ with $\mathcal{Q}'_{T'}$,
	and $\mathcal{U}'_{T'}$ for the intersection of $\mathcal{U}_T$ with $\mathcal{P}'_{T'}$.
	Then, $g'_{T'}$ is a smooth map between smooth manifolds $\mathcal{P}'_{T'}$ and $\mathcal{Q}'_{T'}$,
	and $\mathcal{S}'_\tau$ is a closed smooth manifold of $\mathcal{Q}'_{T'}$.
	The open subset $\mathcal{U}'$ is an open subset in $\mathcal{P}'_{T'}$.
	The manifolds are all canonically oriented (see Remark \ref{nu_J_orientation_remark}), 
	and $g'_{T'}$ is transverse to $\mathcal{S}'_\tau$,
	intersecting at a finite number of points.
	Moreover, as we have assumed $\mathcal{P}'_{T'}\subset\mathcal{U}_T$ (see Lemma \ref{local_data}), 
	implying $\mathcal{U}'_{T'}=\mathcal{P}'_{T'}$,
	the local intersection number of $g'_{T'}$ with $\mathcal{S}'_{T'}$ on $\mathcal{U}'_{T'}$
	is the same as the (global) intersection number.
	Namely, we observe
	$$I\left(g^\nu_T,\mathcal{S}^\nu_\tau;\mathcal{U}^\nu_T\right)_{\mathrm{SU}(2)}
	=I\left(g'_{T'},\mathcal{S}'_{T'};\mathcal{U}'_{T'}\right)
	=I\left(g'_{T'},\mathcal{S}'_{\tau'}\right).$$
	
	Rewrite $r^*_\sigma\colon\mathcal{P}'_{T'}\to \mathcal{Q}'_{T'}$ as $r'_{T'}$,
	for clarity at the momoent.
	Our $(\nu,J)$--equivariant smooth homotopy assumption regarding $g_T$ forces
	$g'_{T'}$ to be smoothly homotopic to $r'_{T'}$.
	Note that the parametrization maps	$\mathcal{P}'_{T'}\to T'$ and $\mathcal{Q}'_{T'}\to T'$
	are bundle projections (see Example \ref{nu_J_example}), and $r'_{T'}$ is a bundle map over $T'$.
	Denote by $r'_\tau\colon\mathcal{P}'_\tau\to\mathcal{Q}'_\tau$
	the fiber map of $r'_{T'}$ at $\tau$,
	where the fibers are precisely the trace $\tau$--slices
	$\mathcal{P}'_\tau=\mathcal{X}^{\mathtt{irr}}_\tau(F_n,\mathrm{SU}(2))$
	and $\mathcal{Q}'_\tau=\mathcal{X}^{\mathtt{irr}}_\tau(WF_n,\mathrm{SU}(2))$.
	Any small perturbation of $r'_\tau$ transverse to $\mathcal{S}'_\tau$
	extends to be some small perturbation of $r'_{T'}$ transverse to $\mathcal{S}'_\tau$,
	(applying a smooth extension followed by a transverse approximation).
	Therefore, we observe
	$$
	I\left(g'_{T'},\mathcal{S}'_\tau\right)
	=I\left(r'_{T'},\mathcal{S}'_\tau\right)
	=I\left(r'_\tau,\mathcal{S}'_\tau\right).$$
		
	The characterization due to Heusener and Kroll \cite[Theorem 1.2]{Heusener--Kroll_abelian} 
	asserts precisely
	$$I\left(r'_\tau,\mathcal{S}'_\tau\right)=\frac{1}{2}\cdot\mathrm{sgn}_K\left(e^{2\cdot\theta\cdot\sqrt{-1}}\right),$$
	assuming $\Delta_K\left(e^{2\cdot\theta\sqrt{-1}}\right)\neq0$.
	
	Joining the above identities, we obtain
	$$I\left(g^\nu_T,\mathcal{S}^\nu_\tau;\mathcal{U}^\nu_T\right)_{\mathrm{SU}(2)}
	%=I\left(g'_{T'},\mathcal{S}_\tau\right)
	=\frac{1}{2}\cdot\mathrm{sgn}_K\left(e^{2\cdot\theta\cdot\sqrt{-1}}\right).$$
	The modulo $2$ reduction of the last identity
	is what we need for completing the proof of Lemma \ref{half_signature_parity}, as explained.
	
	Finally, 
	we point out that asserted identity in Lemma \ref{half_signature_parity} only holds modulo $2$.
	This is	because 
	we have discarded an even number of non-real irreducible $\mathrm{SL}(2,\Complex)$--characters,
	whose signed count depends not only on $K$ and $\tau$, but also on $f_T$ and $\mathcal{U}_T$.
\end{proof}

With the above preparation, we prove Theorem \ref{deform_character} as follows.

Let $K\subset S^3$ be a knot.
Suppose that $e^{2\cdot\theta_0\sqrt{-1}}$ is a complex zero of odd order for
the Alexander polynomial $\Delta_K(t)$,
for some $\theta_0\in(0,\pi)$.
Denote by	$\chi_0\colon\pi_1(S^3\setminus K)\to \Real$ 
the unique reducible $\mathrm{SL}(2,\Real)$--character 
with value $2\cos\theta_0$ at the meridian of $K$.

It suffices to show that $\chi_0$ lies in the closure of
$\mathcal{X}^{\mathtt{irr}}(\pi_1(S^3\setminus K),\mathrm{SL}(2,\Real))$
in $\mathcal{X}(\pi_1(S^3\setminus K),\mathrm{SL}(2,\Real))$.
This is because $\mathcal{X}(\pi_1(S^3\setminus K),\mathrm{SL}(2,\Real))$
is homeomorphic to a locally finite simplicial complex,
such that $\mathcal{X}^{\mathtt{red}}(\pi_1(S^3\setminus K),\mathrm{SL}(2,\Real))$
is a subcomplex (see Subsection \ref{Subsec_analytic}).
If $\chi_0$ lies in the closure of $\mathcal{X}^{\mathtt{irr}}(\pi_1(S^3\setminus K),\mathrm{SL}(2,\Real))$,
then $\chi_0$
lies in the boundary of some simplex with interior in $\mathcal{X}^{\mathtt{irr}}(\pi_1(S^3\setminus K),\mathrm{SL}(2,\Real))$,
giving rise to a continuous path of characters $\chi_s\colon \pi_1(S^3\setminus K)\to \mathrm{SL}(2,\Real)$
extending $\chi_0$ with $\chi_s$ irreducible for all $s\in(0,1]$, as asserted.

To argue by contradiction,
suppose that some neighborhood of $\chi_0$ contains 
no irreducible $\mathrm{SL}(2,\Real)$--characters of $\pi_1(S^3\setminus K)$.
Fix some orientation of $K$ and some braid group element $\sigma\in\mathscr{B}_n$,
for presenting $\pi_1(S^3\setminus K)$, as declared in Lemma \ref{make_braid_presentation}.
Adopt Notations \ref{notation_braid} and \ref{notation_slice}.
We identify $\pi_1(S^3\setminus K)$ with $\Pi_\sigma$.
By pulling back along the group homomorphism $p_\sigma\colon WF_\sigma\to \Pi_\sigma$,
we identify $\mathcal{X}(\Pi_\sigma;\mathrm{SL}(2,\Complex))$ 
with the intersection in $\mathcal{X}({WF}_n,\mathrm{SL}(2,\Complex))$ of the images of
$\mathcal{X}(F_n,\mathrm{SL}(2,\Complex))$
under the proper embeddings $r^*_\sigma$ and $r^*_{\mathrm{id}}$.
		
We construct a tuple $(T,\tau_0,\mathcal{W}_T,\mathcal{V}_T)$ as follows.
Set
	$$\tau_0=2\cos\theta_0,$$ 
Obtain some small round disk neighborhood centered at $\tau_0$,
denoted as
	$$T\subset\Complex\setminus\{\pm2\},$$ 
such that for all $\tau\neq\tau_0$ in $T$, 
and for all $\lambda\in\Complex\setminus\{0\}$ satisfying $\lambda+\lambda^{-1}=\tau$,
$\lambda^2$ is not a zero of $\Delta_K(t)$.
By our assumption to the contrary,
we can take the radius of $T$ to be sufficiently small,
such that $r^*_\sigma(\mathcal{X}^{\mathtt{red}}_T(F_n,\mathrm{SL}(2,\Complex)))$
has some open neighborhood in $\mathcal{X}(WF_n,\mathrm{SL}(2,\Complex))$
of empty intersection with
$\mathcal{V}=r^*_{\mathrm{id}}(\mathcal{X}^{\mathtt{irr}}(F_n,\mathrm{SL}(2,\Real))$.
Denote 
	$\mathcal{P}_T=\mathcal{X}^{\mathtt{irr}}_T\left(F_n,\mathrm{SL}(2,\Complex)\right)$,
and
	$\mathcal{Q}_T=\mathcal{X}^{\mathtt{irr}}_T\left({WF}_n,\mathrm{SL}(2,\Complex)\right)$.
For any $\tau\in T$, denote by $\mathcal{S}_\tau$ the image of 
	$\mathcal{X}^{\mathtt{irr}}_\tau\left(F_n,\mathrm{SL}(2,\Complex)\right)$
under $r^*_{\mathrm{id}}\colon\mathcal{P}_T\to \mathcal{Q}_T$.
View $T$, $\mathcal{P}_T$, and $\mathcal{Q}_T$ as 
	$(\nu,J)$--smooth manifolds,
and $\mathcal{S}_\tau$ as a closed smooth submanifold of $\mathcal{Q}_T$,
according to Example \ref{nu_J_example}.
Set
	$$\mathcal{V}_T=\mathcal{Q}_T	\cap \mathcal{V},$$
and
	$$\mathcal{W}_T=\mathcal{X}_T(F_n,\mathrm{SL}(2,\Complex))\setminus \left((r^*_\sigma)^{-1}(\mathcal{V}_T)\right).$$
By construction,
$\mathcal{W}_T\subset\mathcal{X}_T(F_n,\mathrm{SL}(2,\Complex))$	
is a complex-conjugation invariant, open subset,
containing
	$\mathcal{X}^{\mathtt{red}}_T(F_n,\mathrm{SL}(2,\Complex)\cup\mathcal{X}_T(F_n,\mathrm{SU}(2))$;
	$\mathcal{V}_T\subset\mathcal{Q}_T$ is a $\nu$--invariant, closed subset,
having empty intersection with $r^*_\sigma(\mathcal{W}_T)$.
	
Applying Lemma \ref{local_data}, we obtain some tuple
	$$(f_T,\mathcal{U}_T,D),$$
with respect to $(\sigma,T,\tau_0)$.
We require $(f_T,\mathcal{U}_T,D)$ to satisfy the additional properties therein,
with respect to $\mathcal{W}_T$ and $\mathcal{V}_T$.
Possibly after discarding the components not containing $\tau_0$,
we assume that $D$ is connected.
	
By construction, any real $\tau$ in $D$ satisfies $|\tau|<2$,
so $\tau=2\cos\theta$ holds for some $\theta\in(0,\pi)$.
Moreover, $e^{2\cdot\theta\cdot\sqrt{-1}}$ is not a zero of $\Delta_K(t)$, unless $\tau=\tau_0$.
Since $f_T$ is $(\nu,J)$--equivariantly homotopic to $r^*_\sigma$ (see Lemma \ref{local_data}),
we observe $f_T(\mathcal{X}^{\mathtt{irr}}_T(F_n,\mathrm{SL}(2,\Real)))\subset \mathcal{V}_T$;
meanwhile, 
our construction guarantees 
$f_T(\mathcal{W}_T)\cap\mathcal{V}_T=\emptyset$ and $\mathcal{U}_T\subset\mathcal{W}_T$ (see Lemma \ref{local_data}).
In particular, 
$f^{-1}_T(\mathcal{S}_{\tau})\cap\mathcal{U}_T$ 
must have empty intersection with $\mathcal{X}^{\mathtt{irr}}_T(F_n,\mathrm{SL}(2,\Real))$ in $\mathcal{P}_T$.
	
Consider some real $\tau_-=2\cos\theta_-$ and $\tau_+=2\cos\theta_+$ 
in $D\setminus\{\tau_0\}$,
separated by $\tau_0$ along the real axis, namely, 
$$\tau_-<\tau_0<\tau_+.$$

The local modulo $2$ intersection number
$I_2(f_T,\mathcal{S}_\tau;\mathcal{U}_T)$
is well-defined for all $\tau\in D\setminus\{\tau_0\}$.
It is locally constant as a function in $\tau$,
by standard differential topology \cite[Chapter 5, Theorem 2.1]{Hirsch_book}.
Since $\tau_-$ and $\tau_+$ are connected by some path in $D\setminus\{\tau_0\}$,
we obtain 
$$I_2(f_T,\mathcal{S}_{\tau_-};\mathcal{U}_T)= I_2(f_T,\mathcal{S}_{\tau_+};\mathcal{U}_T),$$
as an identity in $\Integral/2\Integral$.
	
On the other hand,
the signature function jumps by $2$ modulo $4$ when 
the variable $e^{2\cdot\theta\cdot\sqrt{-1}}$ on the complex unit circle 
passes through a zero of odd order for the Alexander polynomial 
\cite[Theorem 2.4]{Gilmer--Livingston_jump}.
This implies
	$$1+\frac12\cdot\mathrm{sgn}_K\left(e^{2\cdot\theta_-\cdot\sqrt{-1}}\right)
	\equiv\frac12\cdot\mathrm{sgn}_K\left(e^{2\cdot\theta_+\cdot\sqrt{-1}}\right)\bmod 2.$$
	
Nothing about the $(\nu,J)$--equivariance of $f_T$ has been used so far.
At this point, we apply Lemma \ref{half_signature_parity}, inferring
$$I_2(f_T,\mathcal{S}_{\tau_\pm};\mathcal{U}_T)\equiv
\frac{1}{2}\cdot\mathrm{sgn}_K\left(e^{2\cdot\theta_{\pm}\cdot\sqrt{-1}}\right)\bmod 2.$$
The last identity contradicts the previous two identities, as desired.

This completes the proof of Theorem \ref{deform_character}.

\section{Inferring an odd-order unitary zero}\label{Sec-order}
In this section, we prove Theorem \ref{main_zero}.

In fact, we prove a more general criterion, as Theorem \ref{criterion_odd_order} below.
For any knot $K\subset S^3$, the Alexander polynomial $\Delta_K\in\Integral[t,t^{-1}]$
satisfies $\Delta_K(t)=\Delta_K(t^{-1})$ and $\Delta_K(1)=1$.
Hence, $\Delta_K$ satisfies the assumptions in Theorem \ref{criterion_odd_order},
and Theorem \ref{main_zero} follows directly from Theorem \ref{criterion_odd_order}.

\begin{theorem}\label{criterion_odd_order}
	Suppose that $P\in\Integral[t,t^{-1}]$ is an integral Laurent polynomial
	of palindrome form $P(t)=a_0+a_1\cdot(t+t^{-1})+\cdots+a_d\cdot(t^d+t^{-d})$,
	such that $a_0$ is odd.
	Then, $P(t)$ has a zero of odd order on the complex unit circle,
	if the following inequality holds for some $j\in\{1,\ldots,d\}$.
	$$|a_j|\geq |a_0|\cdot\cos\left(\frac{\pi}{\lfloor d/j\rfloor +2}\right)$$
\end{theorem}

The rest of this section is devoted to the proof of Theorem \ref{criterion_odd_order}.
We make use the following criterion for the existence of a zero on the complex unit circle,
due to Konvalina and Matache \cite[Theorem 1]{Konvalina--Matache}.

\begin{theorem}[Konvalina--Matache]\label{criterion_KM}
	Suppose that $P\in\Real[t,t^{-1}]$ is a real Laurent polynomial
	of palindrome form $P(t)=a_0+a_1\cdot(t+t^{-1})+\cdots+a_d\cdot(t^d+t^{-d})$.
	Then, $P(t)$ has a zero on the complex unit circle,
	if the following inequality holds for some $j\in\{1,\ldots,d\}$.
	$$|a_j|\geq |a_0|\cdot\cos\left(\frac{\pi}{\lfloor d/j\rfloor +2}\right)$$
\end{theorem}

We strengthen the conclusion of Theorem \ref{criterion_KM}
when the inequality therein is strict.

\begin{lemma}\label{criterion_KM_strict}
	Suppose that $P\in\Real[t,t^{-1}]$ is a real Laurent polynomial
	of palindrome form $P(t)=a_0+a_1\cdot(t+t^{-1})+\cdots+a_d\cdot(t^d+t^{-d})$.
	Then, $P(t)$ has a zero of odd order on the complex unit circle,
	if the following strict inequality holds for some $j\in\{1,\ldots,d\}$.
	$$|a_j|> |a_0|\cdot\cos\left(\frac{\pi}{\lfloor d/j\rfloor +2}\right)$$
\end{lemma}

\begin{proof}
	Note that the condition forces $P\neq0$. 
	%If $a_d=0$, we may ignore the term $a_d\cdot(t^d+t^{-d})$,
	%and the inequality with $d$ replaced by $d-1$	implies the inequality with $d$.
	%Therefore, we assume $a_d\neq0$ without loss of generality.
	%In $\Complex[t,t^{-1}]$, we obtain a unique factorization of $P(t)$
	%as $a_d\cdot t^{-d}\cdot(t-\lambda_1)^{m_1}\cdots(t-\lambda_r)^{m_r}$,
	%where $\lambda_1,\ldots,\lambda_r$ denote the distinct complex zeros of $P(t)$,
	%with respective orders $m_1,\ldots,m_r$ of sum $d$.
	To argue by contradiction, suppose that 
	the zeros of $P(t)$	on the complex unit circle are all of even order.
	Then, we can perturb $P(t)$ into 
	a real Laurent polynomial	of palindromic form, 
	such that the perturbed polynomial has no zeros on the complex unit circle.
	This will lead to a contradiction to Theorem \ref{criterion_KM}.
	
	To be precise,
	we can factorize $P(t)$ uniquely as a product of the form
	$(t+2+t^{-1})^{m_-}\cdot (t-2+t^{-1})^{m_+}\cdot (t+b_1+t^{-1})^{2m_1} \cdot \cdots \cdot
	(t+b_k+t^{-1})^{2m_k}\cdot Q(t)$,
	where $m_{\pm},m_1,\ldots,m_k$ are integers in $\{0,1,2,\cdots\}$,
	and $b_1,\ldots,b_k$ are distinct real coefficients in $(-2,2)$,
	and $Q(t)$ is palindrome in $\Real[t,t^{-1}]$
	with no zeros on the complex unit circle.
	For any $h>1$, we replace the factors $(t\pm2+t^{-1})^{m_{\pm}}$
	with $(t\pm 2h+t^{-1})^{m_\pm}$,
	and replace the factors $(t-b_s+t^{-1})^{2m_s}$ with
	$(t+b_s+ht^{-1})^{m_s}\cdot(ht+b_s+t^{-1})^{m_s}$, for all $s=1,\ldots,k$.
	Then, for any $h>1$,
	the resulting real Laurent polynomial $P_h\in\Real[t,t^{-1}]$
	has no zeros on the complex unit circle.
	
	On the other hand, $P_h$ takes the palindromic form 
	$P_h(t)=a_0(h)+a_1(h)\cdot(t+t^{-1})+\cdots+a_d(h)\cdot(t^d+t^{-d})$,
	where the coefficients $a_j(h)$ 
	all depend continuously on $h$, satisfying $a_j(1)=a_j$.
	By assumption, the strict inequality
	$|a_j(h)|>|a_0(h)|\cdot\cos(\pi/(\lfloor d/j\rfloor+2))$
	holds for some $h>1$ sufficiently close to $1$ 
	and for some $j\in\{1,\ldots,d\}$.
	In this case, Theorem \ref{criterion_KM} implies 
	that $P_h(t)$ has some zero on the complex unit circle,
	which makes a contradiction.
\end{proof}

To prove Theorem \ref{criterion_odd_order},
suppose that $P\in\Integral[t,t^{-1}]$ is an integral Laurent polynomial
of palindrome form $P(t)=a_0+a_1\cdot(t+t^{-1})+\cdots+a_d\cdot(t^d+t^{-d})$,
such that $a_0$ is odd.
Suppose that the inequality $|a_j|\geq |a_0|\cdot\cos(\pi/(\lfloor d/j\rfloor +2))$
holds for some $j\in\{1,\cdots,d\}$.

Note that $\cos(\pi/n)$ is irrational for any integer $n\geq4$ (see \cite{Watkins--Zeitlin}).
If the integer $\lfloor d/j\rfloor +2$ is at least $4$,
the inequality $|a_j|\geq |a_0|\cdot\cos(\pi/(\lfloor d/j\rfloor +2))$
is automatically strict, since $a_0$ and $a_j$ are integers.
Otherwise, the integer $\lfloor d/j\rfloor +2$ can only be equal to $3$.
In this case,
the inequality becomes $|a_j|\geq |a_0|/2$, or equivalently, $2\cdot|a_j|\geq |a_0|$.
Since $a_0$ and $a_j$ are integers, and since $a_0$ is odd,
the inequality is again strict.

Therefore, our assumptions imply a strict inequality
$|a_j|> |a_0|\cdot\cos(\pi/(\lfloor d/j\rfloor +2))$.
Applying Lemma \ref{criterion_KM_strict},
we infer that $P(t)$ has some zero of odd order on the complex unit circle.

This completes the proof of Theorem \ref{criterion_odd_order}.

\section{Further discussion}\label{Sec-discussion}
We briefly discuss a technical point regarding 
the conclusion of Theorem \ref{deform_character}.

In view of Conjecture \ref{LSK_surgery_conjecture},
it would be desirable 
to ask for an asserted path of representations $\rho_s\colon\pi_1(S^3\setminus K)\to \mathrm{SL}(2,\Real)$
as in Corollary \ref{deform_representation},
such that the path $\rho_s$ is 
not constantly trivial evaluated at the longitude of $K$.
When $K$ is a fibered knot, such that the only essential surfaces in the knot complement
are the surface fibers,
a criterion due to Culler and Dunfield
implies that any path $\rho_s$ as asserted in Corollary \ref{deform_representation}
satisfies the additional property as above \cite[Theorem 1.2 and Claim 7.2]{Culler--Dunfield}.
In general, 
the topological condition in Culler and Dunfield's criterion is very restrictive.
See Cremaschi and Yarmola \cite{Cremaschi--Yarmola} for a family of examples 
among two-bridge knots.

Our argument for Theorem \ref{deform_character}
tells little about the trace of $\rho_s$ evaluated at any boundary slope.
In the worst scenario, potentially,
every path $\chi_s$ with the asserted property as in Theorem \ref{deform_character}
could be constant restricted to the peripheral torus subgroup.
Such worrisome phenomena 
have not been observed in any known knot group though.

We pose the following open questions.

\begin{question}\label{ZHS_knot_question}
	For any knot in an integral homology $3$--sphere,
	instead of in $S^3$,
	does the similar assertion as Theorem \ref{deform_character} hold as well?
\end{question}

\begin{question}\label{signature_jump_question}
	For any knot $K\subset S^3$, 
	if $w^2$ is a zero of $\Delta_K$ on the complex unit circle,
	such that the values of $\mathrm{sgn}_K$ near $w^2$ 
	are not equal on the different sides of $w^2$,
	then,	is the abelian $\mathrm{SU}(1,1)$--representation $\rho_w$ 
	as in (\ref{ab_GL1C}) the limit of some continuous family
	of irreducible $\mathrm{SU}(1,1)$--representations?
\end{question}

\begin{question}\label{longitude_rigidity_question}
	Under the assumption of Theorem \ref{deform_character},
	does there exist some continuous path of $\mathrm{SL}(2,\Real)$--characters $\chi_s$ as asserted therein,
	such that the value of $\chi_s$ at the longitude of $K$ is (strictly) less than $2$ for all $s\in(0,1]$?	
\end{question}

See the facts in Theorem \ref{deform_nonabelian_facts}
and the comments in Remark \ref{deform_nonabelian_facts_remark}
for comparison with the above questions.

\bibliographystyle{amsalpha}

\end{document}